\documentclass{amsart}

\usepackage{amssymb,amsmath,latexsym,amsthm}

\theoremstyle{definition}

\newtheorem{definition}{Definition}
\newtheorem{notation}[definition]{Notation}
\newtheorem{example}[definition]{Example}
\newtheorem{remark}[definition]{Remark}

\newtheorem{openproblem}{Open Problem}

\theoremstyle{plain}

\newtheorem{lemma}[definition]{Lemma}
\newtheorem{proposition}[definition]{Proposition}
\newtheorem{theorem}[definition]{Theorem}

\newcommand{\sltwo}{\mathfrak{sl}_2(\mathbb{C})}
\newcommand{\slthree}{\mathfrak{sl}_3(\mathbb{C})}

\begin{document}

\title{Lie invariants in two and three variables}

\author{Murray R. Bremner}

\address{Department of Mathematics and Statistics,
University of Saskatchewan, Canada}

\email{bremner@math.usask.ca}

\author{Jiaxiong Hu}

\address{Department of Mathematics and Statistics,
University of Saskatchewan, Canada}

\email{jih568@mail.usask.ca}

\date{\today}

\noindent

\begin{abstract}
We use computer algebra to determine the Lie invariants of degree $\le 12$ in
the free Lie algebra on two generators corresponding to the natural
representation of the simple 3-dimensional Lie algebra $\sltwo$. We then
consider the free Lie algebra on three generators, and compute the Lie
invariants of degree $\le 7$ corresponding to the adjoint representation of
$\sltwo$, and the Lie invariants of degree $\le 9$ corresponding to the natural
representation of $\slthree$.  We represent the action of $\sltwo$ and
$\slthree$ on Lie polynomials by computing the coefficient matrix with respect
to the basis of Hall words.  We then use algorithms for linear algebra (row
canonical form, Hermite normal form, lattice basis reduction) to compute a
basis of the nullspace.
\end{abstract}

\keywords{Free Lie algebras, invariant Lie polynomials, representation theory
of simple Lie algebras, computer algebra, Hermite normal form.}

\subjclass[2000]{Primary 17B01. Secondary 13A50, 15A72, 16W22, 17B10, 17B40,
17B65.}

\maketitle


\section{Introduction}

The theory of Lie invariants (that is, elements of the free Lie algebra which
are invariant under the action of some group of automorphisms) has been studied
since the pioneering work of Magnus in 1940 and Wever in 1949. Magnus
\cite[page 147]{Magnus} found two Lie invariants, in degrees 2 and 6, for the
action of the general linear group in two dimensions in its natural
representation:
  \[
  [a,b],
  \qquad
  [[a,[a,b]],[b,[a,b]]].
  \]
Wever \cite{Wever1, Wever2} found a Lie invariant in degree 9 for the action of
the general linear group in three dimensions in its natural representation:
  \[
  [[[[a{,}b]{,}[a{,}c]]{,}[[a{,}b]{,}[b{,}c]]]{,}c]
  +
  [[[[b{,}c]{,}[b{,}a]]{,}[[b{,}c]{,}[c{,}a]]]{,}a]
  +
  [[[[c{,}a]{,}[c{,}b]]{,}[[c{,}a]{,}[a{,}b]]]{,}b].
  \]
He also showed that the space of invariants in degree 9 has dimension 4. Wever
observed that the degree $d$ of a Lie invariant for the action of the general
linear group in $q$ dimensions in its natural representation (that is, in the
free Lie algebra on $q$ generators) must be a multiple of $q$; more precisely,
$d = mq$, where $m \ge 2$ unless $q = 2$. He found that there are no Lie
invariants for $q = 2$ generators in degree $d = 4$; for $q = 3$ generators in
degree $d = 6$; and for $q = d \ge 3$. He gave an explicit formula for the
dimension of the space of invariants for $q = 2$ generators and arbitrary
degree $d$. In 1958, Burrow \cite{Burrow1} showed that there is always a Lie
invariant of degree $d = mq$ for all $m \ge 2$ and $q \ge 2$, except in the two
cases noticed by Wever.  In 1967, Burrow \cite{Burrow2} extended Wever's
formula to the case $d = mq$ for all $m \ge 2$ and $q \ge 2$.  A brief survey
of these developments may be found in Reutenauer \cite[\S 8.6.2]{Reutenauer}.
Lie invariants in the natural representation of the general linear group are
important in group theory owing to the connection between free Lie rings and
the lower central series of free groups; see Magnus et al.~\cite{Magnus}.  We
emphasize that all of these references do not consider the Lie invariants
corresponding to any other irreducible representations of the general linear
group.

By the Shirshov-Witt theorem \cite{Shirshov, ShirshovWorks, Witt2}, we know
that every subalgebra $J$ of a free Lie algebra $L$ is also free, and so the
Lie subalgebra of invariants is a free Lie algebra. We are primarily interested
in the primitive invariants; that is, a set of free generators for $J$. If we
have computed the dimensions of the homogeneous subspaces of $J$ in all degrees
$< n$, then we can use the generalized Witt dimension formula of Kang and Kim
\cite{KangKim} and Jurisich \cite{Jurisich} to compute the dimension of the
subspace of non-primitive invariants in degree $n$; that is, the subspace
generated by the invariants of lower degree. The rapid growth of the number of
Lie invariants of degree $mq$ as a function of $m$ suggests that the free Lie
subalgebra of Lie invariants in $q$ generators may not be finitely generated.
So it is most likely impossible to determine all the Lie invariants, or even
all the primitive invariants, even in the simplest case of the natural
representation of the general linear group in two dimensions.

In 1998, Bremner \cite{Bremner} used computer algebra to determine the Lie
invariants (primitive and non-primitive) of degree $\le 10$ in this simplest
case. In this paper we extend those explicit results to degree 12 (Section
\ref{sl2natural}), and compute the dimensions in degree 14. We also consider
Lie invariants in three variables in two cases:
  \begin{itemize}
  \item For the adjoint representation of the general linear group in two
      dimensions (Section \ref{sl2adjoint}), we obtain explicit invariants
      up to degree 7, and dimensions up to degree 12. This seems to be the
      first time that Lie invariants have been considered in a
      representation other than the natural representation.
  \item For the natural representation of the general linear group in three
      dimensions (Section \ref{sl3natural}), we obtain explicit results up
      to degree 9, and partial results in degree 12.  This seems to be the
      first time that an explicit basis has been computed for the space of
      invariants in degree 9.
  \end{itemize}
These are essentially the only cases where a simple Lie algebra has an
irreducible 3-dimensional representation; the only other possibility is the
dual of the natural representation of $\slthree$, and the spaces of invariants
for that case will be linearly isomorphic to the spaces of invariants for the
natural representation.

Most of our results depend on calculations with the computer algebra system
Maple, especially the packages \texttt{LinearAlgebra} and
\texttt{LinearAlgebra[Modular]}. Our calculations depend on computing the row
canonical form of a large integer matrix over the field of rational numbers,
and extracting the canonical basis of the nullspace. However, in a few cases,
we obtain better results by combining an algorithm for the Hermite normal form
of an integer matrix with the LLL algorithm for lattice basis reduction.  See
Bremner and Peresi \cite{BremnerPeresi} for another application of these
methods to problems on free nonassociative algebras.  A comprehensive reference
on modern computer algebra is von zur Gathen and Gerhard \cite{MCA}.


\section{Preliminaries}

\subsection{The Hall basis of a free Lie algebra}

Our main reference for free Lie algebras is Reutenauer \cite{Reutenauer}; see
also Bahturin \cite{Bahturin}. The Hall basis \cite{Hall} of a free Lie algebra
is a subset of the free magma on an alphabet.

\begin{definition} \cite[\S 4]{Reutenauer}
The \emph{free magma} $M(A)$ on a set of generators $A$ is the set of all
nonassociative words on $A$, and can be identified with the set of complete
rooted binary trees with leaves labeled by elements of $A$. Each binary tree
with at least 2 leaf nodes can be written uniquely as $t = (t',t'')$ where $t'$
and $t''$ are its immediate left and right subtrees. The binary operation $M(A)
\times M(A) \rightarrow M(A)$ is the mapping $(t', t'') \rightarrow t$. There
is a canonical map from $M(A)$ onto $A^{*}$, the set of all associative words
on $A$, defined by $f(a) = a$ if $a\in A$ and $f(t) = f(t') f(t'')$ if $t =
(t',t'') \in M(A) \setminus A$. The \emph{degree} of $t$, denoted $\deg(t)$, is
the length of $f(t)$.
\end{definition}

\begin{definition} \cite{Hall}
Let $M(A)$ be the free magma on the set of generators $A$. A subset $H \subset
M(A)$ is called a \emph{Hall basis} if the following conditions hold:
  \begin{enumerate}
  \item $H$ has a total order $<$;
  \item $A\subset H$;
  \item if $t = (t',t'') \in H \setminus A$ then $t', t'' \in H$;
  \item if $t = (t',t'') \in H \setminus A$ then $t' > t''$;
  \item if $t = (t',t'') \in H \setminus A$ then either $t' \in A$ or
      $(t')'' \le t''$.
  \end{enumerate}
The elements of the Hall basis are called \emph{Hall words}.
\end{definition}

The relation between trees and Hall words is given by interpreting each
non-leaf node of a tree $t$ corresponding to a given Hall word as the Lie
bracket of the Hall words corresponding to its subtrees $t'$ and $t''$. We
follow Hall's original method of constructing the Hall basis inductively by
degree.

\begin{definition} \label{totalorderdefinition}
Let $A = \{ a_1, \dots a_q \}$ and define a total order on $A$ by $a_i < a_j$
if and only if $i < j$. We extend this total order inductively to $M(A)$ by
making $<$ agree with $<$ on $A$, and then for $t, u \in M(A)$ we define $t <
u$ if and only if:
  \begin{enumerate}
  \item either $\deg(t) < \deg(u)$, or
  \item $\deg(t) = \deg(u) = 1$ and $t < u$ in $A$, or
  \item $\deg(t)=\deg(u) \ge 2$ and $t' < u'$, or
  \item $\deg(t)=\deg(u) \ge 2$ and $t' = u'$ and $t''< u''$.
  \end{enumerate}
\end{definition}

\begin{theorem} \emph{\textbf{Hall's Theorem.}} \cite[Theorem 3.1]{Hall}
The Hall words on $A$ form a basis for the free Lie ring generated by $A$.
\end{theorem}

The algorithm in the proof of Theorem 3.1 of \cite{Hall} transforms any
nonassociative word in $M(A)$ into a linear combination of Hall words; see
Figure \ref{hallalgorithm}. This process is crucial for us because it
guarantees that we can always obtain a linear combination of Hall words after
applying an element of $\sltwo$ or $\slthree$ to a Hall word.  We emphasize
that a single call to this algorithm is not always sufficient; the terms
$((q,s),p)$ and $((p,s),q)$ may not be Hall words.  (For example, consider $x =
[[[[[b{,}a]{,}a]{,}a]{,}b]{,}a]$.)  The algorithm must be called recursively
until every word is a Hall word.

\begin{figure}
  \begin{itemize}
  \item[] \texttt{Input}: An arbitrary element $x$ in the free magma $M(A)$.
  \item[] \texttt{Output}: A linear combination of Hall words.
  \end{itemize}
  \begin{enumerate}
  \item If $\deg(x) = 1$ then set $\texttt{result} \leftarrow x$.
  \item If $\deg(x) \ge 2$ then $x = (y,z)$:
    \begin{enumerate}
    \item Set $\texttt{result} \leftarrow 0$.
    \item For each term $au$ in $\texttt{HallForm}(y)$ and each term $bv$
        in $\texttt{HallForm}(z)$ do:
      \begin{enumerate}
      \item[] [\emph{apply anticommutativity}]
      \item Set $\epsilon \leftarrow 1$.
      \item If $u = v$ then set $\texttt{newterm} \leftarrow 0$.
      \item If $u < v$ then set $\texttt{newterm} \leftarrow
          (v,u)$ and $\epsilon \leftarrow -1$.
      \item If $u > v$ then set $\texttt{newterm} \leftarrow
          (u,v)$.
      \item If $\texttt{newterm} \ne 0$ then $\texttt{newterm} =
          (r,s)$ with $r > s$:
        \begin{itemize}
        \item If $\deg(r) = 1$ then $\texttt{result} \leftarrow
          \texttt{result} + \epsilon ab\,(r,s)$
        \item If $\deg(r) > 1$ then $r = (p,q)$:
           \begin{itemize}
           \item[] [\emph{apply the Jacobi identity}]
           \item[] If $s \ge q$ then
             \begin{itemize}
             \item[] $\texttt{result} \leftarrow
             \texttt{result} + \epsilon ab\,((p,q),s)$
             \end{itemize}
           else
             \begin{itemize}
             \item[] $\texttt{result} \leftarrow
             \texttt{result} - \epsilon ab\,((q,s),p) + \epsilon ab\,((p,s),q)$.
             \end{itemize}
           \end{itemize}
        \end{itemize}
      \end{enumerate}
    \end{enumerate}
  \item Return \texttt{result}.
  \end{enumerate}
  \caption{Hall's recursive algorithm $\texttt{HallForm}(x)$}
  \label{hallalgorithm}
\end{figure}

\subsection{Irreducible representations of $\sltwo$}

We recall some basic information about the representation theory of Lie
algebras; our reference is Humphreys \cite{Humphreys}. The simple 3-dimensional
Lie algebra $\sltwo$ has the following standard basis:
  \[
  x =
  \left[
  \begin{array}{rr}
  0 & 1 \\
  0 & 0 \\
  \end{array}
  \right],
  \qquad
  h =
  \left[
  \begin{array}{rr}
  1 & 0 \\
  0 & -1 \\
  \end{array}
  \right],
  \qquad
  y =
  \left[
  \begin{array}{rr}
  0 & 0 \\
  1 & 0 \\
  \end{array}
  \right].
  \]
All Lie brackets in $\sltwo$ follow from the standard commutation relations
using bilinearity and anticommutativity:
  \[
  [h,x] = 2x, \qquad [h,y] = -2y, \qquad [x,y] = h.
  \]
Every finite-dimensional representation $M$ of $\sltwo$ is the direct sum of
its weight spaces $M_w$ for the action of $h$; that is, $M_w$ is the eigenspace
for $h$ with eigenvalue $w$. In this paper we are primarily concerned with the
weight spaces $M_0$ and $M_2$ and the linear map $X\colon M_0 \to M_2$ given by
the action of $x$. For every $n \in \mathbb{Z}$, $n \ge 0$, there is an
irreducible representation $V(n)$ of $\sltwo$ with highest weight $n$ and
dimension $n+1$; the action of $\sltwo$ with respect to the basis $\{ \,
v_{n-2i} \, | \, i=0,\hdots,n \, \}$ of weight vectors in $V(n)$ is as follows:
  \begin{align*}
  &
  h.v_{n-2i} = (n{-}2i) v_{n-2i},
  \\
  &
  x.v_n = 0,
  \qquad
  x.v_{n-2i} = (n{-}i{+}1) v_{n-2i+2}
  \;
  (i=1,\hdots,n),
  \\
  &
  y.v_{n-2i} = (i{+}1) v_{n-2i-2}
  \;
  (i=0,\hdots,n-1) \, ,
  \qquad
  y.v_{-n} = 0.
  \end{align*}
Any irreducible finite-dimensional representation of $\sltwo$ is isomorphic to
$V(n)$ for some $n$. Any finite-dimensional representation of $\sltwo$ is
isomorphic to a direct sum of irreducible representations.

\begin{lemma} \label{surjectivelemma}
If $M$ is a finite-dimensional representation of $\sltwo$, then the $x$-action
map $X\colon M_0 \to M_2$ is surjective.
\end{lemma}

\begin{proof}
We first write $M = V(n_1) \oplus \cdots \oplus V(n_k)$ as the direct sum of
irreducible representations.  For any weight $w \in \mathbb{Z}$ we have $M_w =
V(n_1)_w \oplus \cdots \oplus V(n_k)_w$ and so it suffices to prove the claim
for an irreducible representation $M = V(n)$.  If $n = 0$ or $n$ is odd then
$M_2 = \{0\}$ and the claim is vacuous. If $n \ge 2$ is even then let $m \in
M_2$ be arbitrary; thus $m = c v_2$ for some $c \in \mathbb{C}$.  By the action
of $\sltwo$ on $V(n)$ we have
  \[
  x \,.\, ( y \,.\, v_2 )
  =
  x \,.\, \Big( \frac{n}{2} v_0 \Big)
  =
  \frac{n}{2} ( x \,.\, v_0 )
  =
  \frac{n}{2} \Big( \frac{n}{2}+1 \Big) v_2.
  \]
Since $\frac{n}{2} ( \frac{n}{2}+1 ) \ne 0$, this completes the proof.
\end{proof}

This result holds more generally for any simple (finite-dimensional) Lie
algebra $G$ and any finite-dimensional representation $M$.  Consider a simple
root vector $x_i \in G$ and the $x_i$-action map $X_i\colon M_{\mathbf{0}} \to
M_{\mathbf{w}_i}$ where $\mathbf{w}_i$ is the weight of $x_i$.  There exist
$h_i, y_i \in G$ for which the span of $x_i, h_i, y_i$ is a subalgebra
isomorphic to $\sltwo$.  We regard $M$ as a representation of this subalgebra
and apply Lemma \ref{surjectivelemma} to conclude that $X_i$ is surjective. We
use this result in the cases $G = \sltwo$ and $G = \slthree$.


\section{Lie invariants in the natural representation of $\sltwo$}
\label{sl2natural}

We consider the free Lie algebra $L$ generated by the ordered set $A = \{ a, b
\}$ with $a < b$. The generators $a$ and $b$ are the Hall words of degree 1;
they form a basis of the subspace $L_1$.  We regard the 2-dimensional space
$L_1$ as the natural representation of $\sltwo$ by identifying $a$ and $b$ with
column vectors as follows:
  \[
  a = \left[ \begin{array}{r} 1 \\ 0 \end{array} \right],
  \qquad
  b = \left[ \begin{array}{r} 0 \\ 1 \end{array} \right].
  \]
The natural action of $\sltwo$ on $L_1$ by matrix-vector multiplication gives
  \[
  x.a = 0, \quad x.b = a, \quad
  h.a = a, \quad h.b = -b, \quad
  y.a = b, \quad y.b = 0.
  \]
We define the action of $\sltwo$ on the Hall basis of $L_j$ inductively by
degree. For degree 1, the Hall words are $a$ and $b$, and we use the previous
equations. If $[t,u]$ is a Hall word of degree $j \ge 2$, and $D$ is any
element of $\sltwo$, then we use the derivation rule $D.[t,u] = [D.t,u] +
[t,D.u]$.  This action extends linearly to $L_j$.

\begin{definition}
By the \emph{weight} $w$ of the Hall word $t$, we mean its eigenvalue with
respect to the action by $h$: that is, $h.t = wt$.  The \emph{weight space}
$L_j^w$ is the subspace of $L_j$ spanned by the Hall words of weight $w$.
\end{definition}

\begin{notation}
If $t$ is any Hall word then $\#a(t)$ and $\#b(t)$ denote respectively the
number of $a$'s and $b$'s occurring in $t$.
\end{notation}

\begin{lemma}
We have $L_j^w = \mathrm{span} \{ \, t \mid \#a(t) - \#b(t) = w \, \}$.
\end{lemma}

\begin{proof}
By induction on $j$ since the weight of $a$ is 1 and the weight of $b$ is $-1$.
\end{proof}

Our problem is to determine which linear combinations of the Hall words of
degree $j$ are invariant under the natural action of $\sltwo$. That is, we want
to find the elements $Z \in L_j$ which satisfy $D.Z = 0$ for all $D \in
\sltwo$.

\begin{lemma} \label{xactionkernel}
The Lie invariants of degree $j$ for $\sltwo$ in the natural representation are
the kernel of the linear $x$-action map $X \colon L_j^0 \to L_j^2$ defined by
$X(Z) = x.Z$.
\end{lemma}

\begin{proof}
Since $L_j$ is finite-dimensional and $\sltwo$ is semisimple Lie algebra,
Weyl's Theorem implies that $L_j$ is isomorphic to a direct sum of simple
highest weight modules. If $Z \in L_j$ satisfies $D.Z = 0$ for all $D \in
\sltwo$, then clearly $h.Z = 0$, and so $Z$ is in $L_j^0$, the weight space of
weight 0. If $D \ne h$ then it suffices to consider $D = x$ and $D = y$. Since
$x.a = 0$ and $x.b = a$, it is easy to see that the action of $x$ induces a
linear map from $L_j^0$ to $L_j^2$, and the action of $y$ induces a linear map
from $L_j^0$ to $L_j^{-2}$. Thus we need to find all $Z \in L_j^0$ such that
$x.Z = 0$ and $y.Z = 0$. But the equations $h.Z = 0$ and $x.Z = 0$ imply that
$Z$ spans a 1-dimensional simple $\sltwo$-submodule of $L_j^0$ isomorphic to
the highest weight module with highest weight 0, and this implies that $y.Z =
0$. So it suffices to find the $Z \in L_j$ such that $h.Z = 0$ and $x.Z = 0$.
\end{proof}

Bremner \cite{Bremner} determined all the Lie invariants of degree $\le 10$ for
the natural representation of $\sltwo$. By the results of Wever \cite{Wever1,
Wever2}, we know that these invariants can occur only in even degree. We now
extend these computations to degree 12.  We first review the results for degree
$\le 10$: we confirm and simplify previously known results to illustrate our
computational methods.

\begin{lemma}
Any primitive Lie invariant of degree $\le 10$ in the natural representation of
$\sltwo$ is a linear combination of the following six Lie polynomials:
  \begin{align*}
  I_2 &= [b,a],
  \qquad
  I_6 = [[[ba]b][[ba]a]],
  \qquad
  I_{10}^{(1)} = [[[[ba]b][ba]][[[ba]a][ba]]],
  \\
  I_{10}^{(2)}
  &=
  [[[[[ba]a]a]b][[[ba]b][ba]]]
  - 2 [[[[[ba]a]b]b][[[ba]a][ba]]]
  - 3 [[[[[ba]a]b]b][[[[ba]a]a]b]]
  \\
  &\quad
  + [[[[[ba]b]b]b][[[[ba]a]a]a]],
  \\
  I_{10}^{(3)}
  &=
  [[[[[ba]a]a][ba]][[[ba]b]b]]
  - 2 [[[[[ba]a]b][ba]][[[ba]a]b]]
  + [[[[[ba]b]b][ba]][[[ba]a]a]],
  \\
  I_{10}^{(4)}
  &=
  [[[[ba]b][[ba]a]][[[ba]a]b]]
  + [[[[[ba]a]a][[ba]b]][[ba]b]]
  - 2 [[[[[ba]a]b][[ba]a]][[ba]b]]
  \\
  &\quad
  + [[[[[ba]b]b][[ba]a]][[ba]a]].
  \end{align*}

\end{lemma}

\begin{proof}
For degree 2, there is only one Hall word, namely $I_2 = [b,a]$, and we easily
verify that this is an invariant:
  \[
  x.[b,a] = [x.b,a] + [b,x.a] = [a,a] + [b,0] = 0.
  \]
For degree 4, there is only one Hall word of weight 0, namely $[[[b,a],a],b]$,
and only one Hall word of weight 2, namely $[[[b,a],a],a]$.  We calculate
  \begin{align*}
  x.[[[b,a],a],b]
  &=
  [[[x.b,a],a],b] +
  [[[b,x.a],a],b] +
  [[[b,a],x.a],b] +
  [[[b,a],a],x.b]
  \\
  &=
  [[[a,a],a],b] +
  [[[b,0],a],b] +
  [[[b,a],0],b] +
  [[[b,a],a],a]
  \\
  &=
  0 + 0 + 0 + [[[b,a],a],a].
  \end{align*}
Thus $x.[[[b,a],a],b] \ne 0$ and so there is no invariant in degree 4. (This
also follows from the surjectivity of the $x$-action map; see Lemma
\ref{surjectivelemma}.) For degree 6, we have three Hall words of weight 0, and
two Hall words of weight 2:
  \[
  [[[ba]b][[ba]a]], \;\;
  [[[[ba]a]b][ba]], \;\;
  [[[[[ba]a]a]b]b]; \qquad
  [[[[ba]a]a][ba]], \;\;
  [[[[[ba]a]a]a]b].
  \]
(From now on we omit the commas in Hall words.) We calculate
  \begin{align*}
  &
  x . [[[ba]b][[ba]a]] = 0,
  \qquad
  x . [[[[ba]a]b][ba]] = [[[[b a] a] a] [b a]],
  \\
  &
  x . [[[[[ba]a]a]b]b] = [[[[ba]a]a][ba]] + 2 [[[[[ba]a]a]a]b].
  \end{align*}
Hence the kernel of the $x$-action has dimension 1 and basis $I_6 =
[[[ba]b][[ba]a]]$.  (The invariants of degree $\le 6$ were first found by
Magnus \cite{Magnus}.) For degree 8, we have 8 Hall words of weight 0, and 7 of
weight 2:
  \begin{alignat*}{4}
  &[[[[ba]b]b][[[ba]a]a]], &\quad
  &[[[[ba]a][ba]][[ba]b]], &\quad
  &[[[[ba]b][ba]][[ba]a]], &\quad
  &[[[[[ba]a]a]b][[ba]b]],
  \\
  &[[[[[ba]a]b]b][[ba]a]], &\quad
  &[[[[[ba]a]b][ba]][ba]], &\quad
  &[[[[[[ba]a]a]b]b][ba]], &\quad
  &[[[[[[[ba]a]a]a]b]b]b];
  \\
  &[[[[ba]a]b][[[ba]a]a]], &\quad
  &[[[[ba]a][ba]][[ba]a]], &\quad
  &[[[[[ba]a]a]a][[ba]b]], &\quad
  &[[[[[ba]a]a]b][[ba]a]],
  \\
  &[[[[[ba]a]a][ba]][ba]], &\quad
  &[[[[[[ba]a]a]a]b][ba]], &\quad
  &[[[[[[[ba]a]a]a]a]b]b].
  \end{alignat*}
Using these ordered bases, we compute the matrix $[X]$ representing the
$x$-action map $X\colon L_8^0 \to L_8^2$:
  \[
  [X]
  =
  \left[
  \begin{array}{rrrrrrrr}
  2 & . & . & . & . & . & . & . \\
  . & 1 & 1 & . & 1 & . & . & . \\
  . & . & . & 1 & . & . & . & 2 \\
  . & . & . & 1 & 2 & . & . & . \\
  . & . & . & . & . & 1 & 1 & . \\
  . & . & . & . & . & . & 2 & 3 \\
  . & . & . & . & . & . & . & 3
  \end{array}
  \right]
  \]
This matrix has rank 7, and so its nullspace has dimension 1; a basis for the
nullspace consists of the invariant
  \[
  [[[[ba]a][ba]][[ba]b]] - [[[[ba]b][ba]][[ba]a]].
  \]
But this is the Hall form of $[ I_2, I_6 ] = [ [ba] [[[ba]b][[ba]a]] ]$, the
Lie bracket of the invariants of degree 2 and 6, so there are no primitive
invariants in degree 8. For degree 10, there are 25 words of weight 0, and 20
words of weight 2. The $x$-action matrix $[X]$ has size $20 \times 25$; using
Maple we find that its rank is 20, and so its nullspace has dimension 5. From
the row canonical form of $[X]$ we obtain the canonical basis of the space of
invariants:
  \begin{align*}
  J_{10}^{(1)}
  &=
  [[[[ba]b][ba]][[[ba]a][ba]]],
  \\
  J_{10}^{(2)}
  &=
  [[[[[ba]a]a]b][[[ba]b][ba]]]
  - 2 [[[[[ba]a]b]b][[[ba]a][ba]]]
  - 3 [[[[[ba]a]b]b][[[[ba]a]a]b]]
  \\
  &\quad
  + [[[[[ba]b]b]b][[[[ba]a]a]a]],
  \\
  J_{10}^{(3)}
  &=
  [[[[[ba]a]a][ba]][[[ba]b]b]]
  - 2 [[[[[ba]a]b][ba]][[[ba]a]b]]
  + [[[[[ba]b]b][ba]][[[ba]a]a]],
  \\
  J_{10}^{(4)}
  &=
  [[[[ba]b][[ba]a]][[[ba]a]b]]
  + [[[[[ba]a]a][[ba]b]][[ba]b]]
  - 2 [[[[[ba]a]b][[ba]a]][[ba]b]]
  \\
  &\quad
  + [[[[[ba]b]b][[ba]a]][[ba]a]],
  \\
  J_{10}^{(5)}
  &=
  - [[[[[ba]a][ba]][ba]][[ba]b]]
  + [[[[[ba]b][ba]][ba]][[ba]a]].
  \end{align*}
From the primitive invariants $I_2$ and $I_6$ in degrees 2 and 6 we can obtain
only one invariant in degree 10, namely $[I_2,[I_2,I_6]]$. The Hall form of
this invariant is
  \[
  K_{10}
  =
  2 [[[[ba]b][ba]][[[ba]a][ba]]]]
  - [[[[[ba]a][ba]][ba]][[ba]b]]]
  + [[[[[ba]b][ba]][ba]][[ba]a]]].
  \]
Hence there is a 4-dimensional space of primitive invariants in degree 10: we
can choose any subspace, of the 5-dimensional space with basis $\{
J_{10}^{(1)}, \dots, J_{10}^{(5)} \}$, which is complementary to the
1-dimensional span of $K_{10}$. In particular, we easily check that the
elements $\{ K_{10}, J_{10}^{(1)}, \dots, J_{10}^{(4)} \}$ are linearly
independent, and so we may take $\{ J_{10}^{(1)}, \dots, J_{10}^{(4)} \}$ as
the canonical set of primitive invariants in degree 10: these are the
invariants $I_{10}^{(1)}, \dots, I_{10}^{(4)}$ in the statement of the Lemma.
(The invariants of degree 10 were first found by Bremner \cite{Bremner} in a
slightly more complicated form.)
\end{proof}

We now extend these results to degree 12.

\begin{theorem}
Any linear combination of the following four Lie polynomials is a primitive Lie
invariant of degree 12 in the natural representation of $\sltwo$:
  \allowdisplaybreaks
  \begin{align*}
  I_{12}^{(1)}
  &=
  [[[[[ba]a][ba]][ba]][[[ba]b][ba]]]
  -  [[[[[ba]b][ba]][ba]][[[ba]a][ba]]],
  \\
  I_{12}^{(2)}
  &=
  [[[[[[ba]a][ba]][ba]][ba]][[ba]b]]
  -  [[[[[[ba]b][ba]][ba]][ba]][[ba]a]],
  \\
  I_{12}^{(3)}
  &=
  [[[[[[ba]a]a][ba]][ba]][[[ba]b]b]]
  -2  [[[[[[ba]a]b][ba]][ba]][[[ba]a]b]]
  \\
  &\quad
  +  [[[[[[ba]b]b][ba]][ba]][[[ba]a]a]],
  \\
  I_{12}^{(4)}
  &=
  [[[[[ba]a]b][ba]][[[ba]b][[ba]a]]]
  -  [[[[[[ba]a]a][ba]][[ba]b]][[ba]b]]
  \\
  &\quad
  +2  [[[[[[ba]a]b][ba]][[ba]a]][[ba]b]]
  -  [[[[[[ba]b]b][ba]][[ba]a]][[ba]a]].
  \end{align*}
\end{theorem}

\begin{proof}
There are 335 Hall words in degree 12, with 75 of weight 0 and 66 of weight 2.
We now use Maple for the following calculations:
  \begin{enumerate}
  \item
  compute the action of $x$ on each of the Hall words $w$ of weight 0;
  \item compute the Hall form of each resulting word of weight 2;
  \item
  collect and sort the Hall words of weight 2 appearing in $x.w$;
  \item
  construct the matrix $[X]$ representing the linear map
  $X\colon L_{12}^0 \to L_{12}^2$;
  \item
  compute the row canonical form of the $66 \times 75$ matrix $[X]$;
  \item
  extract the canonical basis of the nullspace of $[X]$;
  \item
  sort the canonical basis vectors by increasing Euclidean norm.
  \end{enumerate}
The rank of $[X]$ is 66, and so the nullspace has dimension 9. The canonical
basis of the subspace of invariants in degree 12 consists of the following
elements:
  \allowdisplaybreaks
  \begin{align*}
  J_{12}^{(1)}
  &=
  -  [[[[[ba]a][ba]][ba]][[[ba]b][ba]]]
  +  [[[[[ba]b][ba]][ba]][[[ba]a][ba]]],
  \\
  J_{12}^{(2)}
  &=
  -  [[[[[[ba]a][ba]][ba]][ba]][[ba]b]]
  +  [[[[[[ba]b][ba]][ba]][ba]][[ba]a]],
  \\
  J_{12}^{(3)}
  &=
  [[[[[ba]a]a][[ba]b]][[[ba]b][ba]]]
  -  [[[[[ba]a]b][[ba]a]][[[ba]b][ba]]]
  \\
  &\quad
  -  [[[[[ba]a]b][[ba]b]][[[ba]a][ba]]]
  +  [[[[[ba]b]b][[ba]a]][[[ba]a][ba]]],
  \\
  J_{12}^{(4)}
  &=
  [[[[[ba]a][ba]][[ba]a]][[[ba]b]b]]
  -  [[[[[ba]a][ba]][[ba]b]][[[ba]a]b]]
  \\
  &\quad
  -  [[[[[ba]b][ba]][[ba]a]][[[ba]a]b]]
  +  [[[[[ba]b][ba]][[ba]b]][[[ba]a]a]],
  \\
  J_{12}^{(5)}
  &=
  [[[[[[ba]a]a][ba]][ba]][[[ba]b]b]]
  -2  [[[[[[ba]a]b][ba]][ba]][[[ba]a]b]]
  \\
  &\quad
  +  [[[[[[ba]b]b][ba]][ba]][[[ba]a]a]],
  \\
  J_{12}^{(6)}
  &=
  -  [[[[[ba]a]b][ba]][[[ba]b][[ba]a]]]
  +  [[[[[[ba]a]a][ba]][[ba]b]][[ba]b]]
  \\
  &\quad
  -2  [[[[[[ba]a]b][ba]][[ba]a]][[ba]b]]
  +  [[[[[[ba]b]b][ba]][[ba]a]][[ba]a]],
  \\
  J_{12}^{(7)}
  &=
  [[[[[ba]a][ba]][[ba]a]][[[ba]b]b]]
  +  [[[[[ba]a][ba]][[ba]b]][[[ba]a]b]]
  \\
  &\quad
  -  [[[[[ba]b][ba]][[ba]a]][[[ba]a]b]]
  -  [[[[[[ba]a]a]a][[ba]b]][[[ba]b]b]]
  \\
  &\quad
  +  [[[[[[ba]a]a]b][[ba]a]][[[ba]b]b]]
  + 2  [[[[[[ba]a]a]b][[ba]b]][[[ba]a]b]]
  \\
  &\quad
  -2  [[[[[[ba]a]b]b][[ba]a]][[[ba]a]b]]
  -  [[[[[[ba]a]b]b][[ba]b]][[[ba]a]a]]
  \\
  &\quad
  +  [[[[[[ba]b]b]b][[ba]a]][[[ba]a]a]],
  \\
  J_{12}^{(8)}
  &=
  -  [[[[[ba]a]a][[ba]a]][[[[ba]b]b]b]]
  +  [[[[[ba]a]a][[ba]b]][[[ba]b][ba]]]
  \\
  &\quad
  +  [[[[[ba]a]a][[ba]b]][[[[ba]a]b]b]]
  +2  [[[[[ba]a]b][[ba]a]][[[[ba]a]b]b]]
  \\
  &\quad
  -2  [[[[[ba]a]b][[ba]b]][[[ba]a][ba]]]
  -2  [[[[[ba]a]b][[ba]b]][[[[ba]a]a]b]]
  \\
  &\quad
  -  [[[[[ba]b]b][[ba]a]][[[[ba]a]a]b]]
  +  [[[[[ba]b]b][[ba]b]][[[[ba]a]a]a]],
  \\
  J_{12}^{(9)}
  &=
   2  [[[[[ba]a][ba]][ba]][[[[ba]a]b]b]]
  -  [[[[[ba]b][ba]][ba]][[[[ba]a]a]b]]
  \\
  &\quad
  -  [[[[[[ba]a]a]a][ba]][[[[ba]b]b]b]]
  +  [[[[[[ba]a]a]b][ba]][[[ba]b][ba]]]
  \\
  &\quad
  +3  [[[[[[ba]a]a]b][ba]][[[[ba]a]b]b]]
  -2  [[[[[[ba]a]b]b][ba]][[[ba]a][ba]]]
  \\
  &\quad
  -3  [[[[[[ba]a]b]b][ba]][[[[ba]a]a]b]]
  +  [[[[[[ba]b]b]b][ba]][[[[ba]a]a]a]].
  \end{align*}
We compute the non-primitive invariants in degree 12 by taking the Lie bracket
of the invariant $I_2$ in degree 2 with the basis invariants $J_{10}^{(1)}, \dots,
J_{10}^{(5)}$ in degree 10, and reducing all the words to their Hall form:
  \allowdisplaybreaks
  \begin{align*}
  [ I_2, J_{10}^{(1)} ]
  \to
  K_{12}^{(1)}
  &=
  [[[[[ba]a][ba]][ba]][[[ba]b][ba]]]
  -  [[[[[ba]b][ba]][ba]][[[ba]a][ba]]],
  \\
  [ I_2, J_{10}^{(2)} ]
  \to
  K_{12}^{(2)}
  &=
  -2  [[[[[ba]a][ba]][ba]][[[[ba]a]b]b]]
  +  [[[[[ba]b][ba]][ba]][[[[ba]a]a]b]]
  \\
  &\quad
  +  [[[[[[ba]a]a]a][ba]][[[[ba]b]b]b]]
  -  [[[[[[ba]a]a]b][ba]][[[ba]b][ba]]]
  \\
  &\quad
  -3  [[[[[[ba]a]a]b][ba]][[[[ba]a]b]b]]
  +2  [[[[[[ba]a]b]b][ba]][[[ba]a][ba]]]
  \\
  &\quad
  +3  [[[[[[ba]a]b]b][ba]][[[[ba]a]a]b]]
  -  [[[[[[ba]b]b]b][ba]][[[[ba]a]a]a]],
  \\
  [ I_2, J_{10}^{(3)} ]
  \to
  K_{12}^{(3)}
  &=
  -  [[[[[[ba]a]a][ba]][ba]][[[ba]b]b]]
  +2  [[[[[[ba]a]b][ba]][ba]][[[ba]a]b]]
  \\
  &\quad
   -  [[[[[[ba]b]b][ba]][ba]][[[ba]a]a]],
  \\
  [ I_2, J_{10}^{(4)} ]
  \to
  K_{12}^{(4)}
  &=
  [[[[[ba]a]b][ba]][[[ba]b][[ba]a]]]
  -2  [[[[[ba]a]a][[ba]b]][[[ba]b][ba]]]
  \\
  &\quad
  +2  [[[[[ba]a]b][[ba]a]][[[ba]b][ba]]]
  +2  [[[[[ba]a]b][[ba]b]][[[ba]a][ba]]]
  \\
  &\quad
  -2  [[[[[ba]b]b][[ba]a]][[[ba]a][ba]]]
  +  [[[[[ba]a][ba]][[ba]a]][[[ba]b]b]]
  \\
  &\quad
  -  [[[[[ba]a][ba]][[ba]b]][[[ba]a]b]]
  -  [[[[[ba]b][ba]][[ba]a]][[[ba]a]b]]
  \\
  &\quad
  +  [[[[[ba]b][ba]][[ba]b]][[[ba]a]a]]
  -  [[[[[[ba]a]a][ba]][[ba]b]][[ba]b]]
  \\
  &\quad
  +2  [[[[[[ba]a]b][ba]][[ba]a]][[ba]b]]
  -  [[[[[[ba]b]b][ba]][[ba]a]][[ba]a]],
  \\
  [ I_2, J_{10}^{(5)} ]
  \to
  K_{12}^{(5)}
  &=
  [[[[[ba]a][ba]][ba]][[[ba]b][ba]]]
  -  [[[[[ba]b][ba]][ba]][[[ba]a][ba]]]
  \\
  &\quad
  +  [[[[[[ba]a][ba]][ba]][ba]][[ba]b]]
  -  [[[[[[ba]b][ba]][ba]][ba]][[ba]a]].
  \end{align*}
Hence there is a 4-dimensional space of primitive invariants in degree 12: we
can choose any subspace, of the 9-dimensional space with basis $\{J_{12}^{(1)},
\dots, J_{12}^{(9)} \}$, which is complementary to the 5-dimensional span of
$\{ K_{12}^{(1)}, \dots, K_{12}^{(5)} \}$. In particular, we easily check that
the elements $\{ K_{12}^{(1)}, \dots, K_{12}^{(5)}, J_{12}^{(1)}, J_{12}^{(2)},
J_{12}^{(5)}, J_{12}^{(6)} \}$ are linearly independent, and so we may take $\{
J_{12}^{(1)}, J_{12}^{(2)}, J_{12}^{(5)}, J_{12}^{(6)} \}$ as the canonical set
of primitive invariants in degree 12. These (with sign changes) are the
invariants $I_{12}^{(1)}$, \dots, $I_{12}^{(4)}$ in the statement of the
Theorem.  This completes the proof.
\end{proof}

We extended these calculations to degree 14, and obtained a 33-dimensional
space of invariants, with a 10-dimensional subspace of non-primitive invariants
spanned by Lie brackets of invariants of lower degree.  Hence there is a
23-dimensional space of primitive invariants in degree 14; see Hu \cite{Hu} for
details.


\section{Lie invariants in the adjoint representation of $\sltwo$}
\label{sl2adjoint}

In this section, we consider for the first time the Lie invariants for a simple
finite-dimensional Lie algebra in a representation other than the natural
representation.  We study the simplest case, namely the adjoint representation
of $\sltwo$.

We consider the free Lie algebra $L$ generated by the ordered set $A = \{ a, b,
c \}$ with $a < b < c$. The generators $a$, $b$ and $c$ are the Hall words of
degree 1; they form a basis of the subspace $L_1$.  We regard the 3-dimensional
space $L_1$ as the adjoint representation of $\sltwo$ by identifying $a$, $b$,
$c$ with $x$, $h$, $y$ respectively. The matrices for adjoint representation of
$\sltwo$ are as follows:
  \[
  \mathrm{ad}(x)
  =
  \left[
  \begin{array}{rrr}
  0 &\!\!\! -2 & 0 \\
  0 &\!\!\!  0 & 1 \\
  0 &\!\!\!  0 & 0 \\
  \end{array}
  \right],
  \quad
  \mathrm{ad}(h)
  =
  \left[
  \begin{array}{rrr}
  2 & 0 &\!\!\!  0 \\
  0 & 0 &\!\!\!  0 \\
  0 & 0 &\!\!\! -2 \\
  \end{array}
  \right],
  \quad
  \mathrm{ad}(y)
  =
  \left[
  \begin{array}{rrr}
  0 & 0 & 0 \\
  -1 & 0 & 0 \\
  0 & 2 & 0 \\
  \end{array}
  \right].
  \]
This gives the action of $\sltwo$ on $L_1$:
  \[
  \begin{array}{lll}
  x.a =  0, &\quad x.b = -2a, &\quad x.c =   b,
  \\
  h.a = 2a, &\quad h.b =   0, &\quad h.c = -2c,
  \\
  y.a = -b, &\quad y.b =  2b, &\quad y.c =   0.
  \end{array}
  \]

\begin{lemma}
We have $L_j^w = \mathrm{span} \{ \, t \mid 2\,\#a(t) - 2\,\#c(t) = w \, \}$.
\end{lemma}

\begin{proof}
The $h$-action shows that $a$, $b$, $c$ have weights 2, 0, $-2$ respectively.
\end{proof}

\begin{lemma}
The Lie invariants of degree $j$ for $\sltwo$ in the adjoint representation are
the kernel of the linear $x$-action map $X \colon L_j^0 \to L_j^2$ defined by
$X(Z) = x.Z$.
\end{lemma}

\begin{proof}
Similar to the proof of Lemma \ref{xactionkernel}.
\end{proof}

We remark that the dimension formulas of Wever \cite{Wever1, Wever2} and Burrow
\cite{Burrow1, Burrow2} do not apply in this case, since they apply only to the
natural representation. The Lie invariants for the adjoint representation
appear not to have been studied before.

\begin{lemma}
There are no Lie invariants for the adjoint representation of $\sltwo$ in
degree $\le 4$.
\end{lemma}

\begin{proof} \label{adjointdegree4}
In degree 1, there is one Hall word $b$ of weight 0 and one Hall word $a$ of
weight 2.  Since $x.b = -2a$, the kernel of the $x$-action is zero. In degree
2, there is one Hall word $[c,a]$ of weight 0 and one Hall word $[b,a]$ of
weight 2.  Since $x.[c,a] = [b,a]$, the kernel of the $x$-action is zero. In
degree 3, there are two Hall words of weight 0, namely $[[b,a],c]$ and
$[[c,a],b]$, and two of weight 2, namely $[[b,a],b]$ and $[[c,a],a]$.  We have
$x . [[b,a],c] = [[b,a],b]$ and $x . [[c,a],b] = [[b,a],b] - 2 [[c,a],a]$, and
again the kernel of the $x$-action is zero. In degree 4, there are four Hall
words of weight 0, and four of weight 2,
  \begin{align*}
  &
  [[c, b], [b, a]], \quad
  [[[b, a], b], c], \quad
  [[[c, a], a], c], \quad
  [[[c, a], b], b];
  \\
  &
  [[c, a], [b, a]], \quad
  [[[b, a], a], c], \quad
  [[[b, a], b], b], \quad
  [[[c, a], a], b].
  \end{align*}
The $x$-action is given by these equations:
  \begin{align*}
  &
  x . [[c{,} b]{,} [b{,} a]]
  =
  -2  [[c{,} a]{,} [b{,} a]]{,}
  \qquad
  x . [[[b{,} a]{,} b]{,} c]
  =
  -2  [[[b{,} a]{,} a]{,} c] + [[[b{,} a]{,} b]{,} b]{,}
  \\
  &
  x . [[[c{,} a]{,} a]{,} c]
  =
  [[[b{,} a]{,} a]{,} c] + [[[c{,} a]{,} a]{,} b]{,}
  \\
  &
  x . [[[c{,} a]{,} b]{,} b]
  =
  -2  [[c{,} a]{,} [b{,} a]] + [[[b{,} a]{,} b]{,} b] - 4  [[[c{,} a]{,} a]{,} b].
  \end{align*}
The $x$-action matrix has full rank, and so there are no invariants.  We remark
that these results also follow from Lemma \ref{surjectivelemma} on the
surjectivity of the $x$-action map, since in every case the numbers of Hall
words for weights 0 and 2 are equal.
\end{proof}

\begin{theorem} \label{adjointdegree5}
For the adjoint representation of $\sltwo$, the space of invariants in degree 5
has dimension 1 and basis
  \[
  I_5
  =
  [ [ [b a] b ] [c b] ]
  +
  4
  [ [ [b a] c ] [c a] ]
  +
  2
  [ [ [c a] a ] [c b] ]
  -
  2
  [ [ [c a] b ] [c a] ]
  -
  2
  [ [ [c a] c ] [b a] ]
  -
  [ [ [c b] b ] [b a] ].
  \]
\end{theorem}

\begin{proof}
In degree 5, there are 10 Hall words of weight 0,
  \begin{align*}
  &
  [[[ba]b][cb]], \quad
  [[[ba]c][ca]], \quad
  [[[ca]a][cb]], \quad
  [[[ca]b][ca]], \quad
  [[[ca]c][ba]],
  \\
  &
  [[[cb]b][ba]], \quad
  [[[[ba]a]c]c], \quad
  [[[[ba]b]b]c], \quad
  [[[[ca]a]b]c], \quad
  [[[[ca]b]b]b],
  \end{align*}
and 9 of weight 2,
  \begin{align*}
  &
  [[[ba]a][cb]], \quad
  [[[ba]b][ca]], \quad
  [[[ba]c][ba]], \quad
  [[[ca]a][ca]], \quad
  [[[ca]b][ba]],
  \\
  &
  [[[[ba]a]b]c], \quad
  [[[[ba]b]b]b], \quad
  [[[[ca]a]a]c], \quad
  [[[[ca]a]b]b].
  \end{align*}
We calculate the $x$-action on the words of weight 0, and express the results
as linear combinations of the words of weight 2.  We obtain the following
matrix $[X]$ representing the map $X \colon L_5^0 \longrightarrow L_5^2$:
  \[
  [X]
  =
  \left[
  \begin{array}{rrrrrrrrrr}
  -2 & \phantom{-}. & 1 & . & \phantom{-}. & . & \phantom{-}1 & . & . & . \\
  -2 & 1 &  . & 1 & . & . & . & . & . & 4 \\
   . & 1 &  . & . & 1 & 2 & . & . & . & . \\
   . & . & -2 & -2 & . & . & . & . & . & . \\
   . & . &  . & 1 & 1 & -4 & . & . & . & -6 \\
   . & . &  . & . & . & . & 2 & -4 & 1 & . \\
   . & . &  . & . & . & . & . & 1 & . & 1\\
   . & . &  . & . & . & . & . & . & -2 & . \\
   . & . &  . & . & . & . & . & . & 1 & -6
  \end{array}
  \right]
 \]
We compute the row canonical form of this matrix, and find that the rank is 9,
so the nullspace has dimension 1.  A basis of the nullspace is the invariant
$I_5$.
\end{proof}

\begin{example}
We can give an explicit direct proof that $I_5$ is an invariant.  We calculate
as follows, using Hall's algorithm to simplify the results:
  \allowdisplaybreaks
  \begin{align*}
  x.[[[ba]b][cb]] &= -2[[[ba]a][cb]] -2[[[ba]b][ca]],
  \\
  x.[[[ba]c][ca]] &=   [[[ba]b][ca]]  +[[[ba]c][ba]],
  \\
  x.[[[ca]a][cb]] &=   [[[ba]a][cb]] -2[[[ca]a][ca]],
  \\
  x.[[[ca]b][ca]] &=   [[[ba]b][ca]] -2[[[ca]a][ca]] +[[[ca]b][ba]],
  \\
  x.[[[ca]c][ba]] &=   [[[ba]c][ba]]  +[[[ca]b][ba]],
  \\
  x.[[[cb]b][ba]] &=  2[[[ba]c][ba]] -4[[[ca]b][ba]].
  \end{align*}
The linear combination with coefficients $1, 4, 2, -2, -2, -1$ gives 0 as
required.
\end{example}

\begin{theorem}
For the adjoint representation of $\sltwo$, the space of invariants in degree 6
has dimension 1 and basis
  \begin{align*}
  I_6
  &=
  4 [[[cb]b][ba]]
  -2 [[[cb][ba]][ca]]
  -4 [[[[ba]a]c][cb]]
  -  [[[[ba]b]b][cb]]
  +4 [[[[ba]c]c][ba]]
  \\
  &\quad
  -8 [[[[ca]a]c][ca]]
  -2 [[[[ca]b]b][ca]]
  -4 [[[[ca]b]c][ba]]
  -  [[[[cb]b]b][ba]].
  \end{align*}
\end{theorem}

\begin{proof}
In degree 6, there are 22 Hall words of weight 0 and 21 of weight 2. The
$x$-action has a 1-dimensional kernel with $I_6$ as its basis.
\end{proof}

\begin{remark}
Since there are no invariants in degree $\le 4$, only one invariant in degree 5
and only one in degree 6, it follows that the smallest degree for a
non-primitive invariant is 11. That is, every invariant in degree $\le 10$ is
primitive.
\end{remark}

In degree 7, there are 56 Hall words of weight 0, and 51 of weight 2.  Hence
the $x$-action matrix $[X]$ has size $51 \times 56$; we use Maple to compute
that its rank is 51, and so its nullspace has dimension 5. If we compute the
row canonical form of $[X]$ over $\mathbb{Q}$, most of the entries of the RCF
are integers, but some are fractions with denominator 2. We extract the
canonical basis of the nullspace, and clear denominators where necessary to
obtain a basis for the nullspace consisting of vectors with relatively prime
integer components. The squared Euclidean norms of these vectors are 15, 23,
514, 690, 3218.  We can obtain better results by computing an integral basis
for the nullspace using the Hermite normal form combined with the LLL algorithm
for lattice basis reduction.

\begin{definition}
The $m \times n$ matrix $H$ over $\mathbb{Z}$ is in \emph{Hermite Normal Form
(HNF)} if there exists $r \in \mathbb{Z}$, $0 \le r \le m$, and integers $1 \le
j_1 < j_2 < \cdots < j_r \le n$ such that
  \begin{itemize}
  \item $H_{ij} = 0$ for $1 \le i \le r$ and $1 \le j < j_i$,
  \item $H_{i,j_i} \ge 1$ for $1 \le i \le r$,
  \item $0 \le H_{k,j_i} < H_{i,j_i}$ for $1 \le i \le r$ and $1 \le k <
      i$,
  \item $H_{ij} = 0$ for $r{+}1 \le i \le m$ and $1 \le j \le n$.
  \end{itemize}
\end{definition}

\begin{proposition}
If $A$ is an $m \times n$ matrix over $\mathbb{Z}$, then there is a unique $m
\times n$ matrix $H$ over $\mathbb{Z}$ in HNF such that $UA = H$ for some $m
\times m$ matrix $U$ over $\mathbb{Z}$ with $\det(U) = \pm 1$. (The matrix $U$
is in general not unique.)
\end{proposition}

\begin{proof}
Adkins and Weintraub \cite[\S 5.2]{AdkinsWeintraub}.
\end{proof}

\begin{proposition}
Let $A$ be an $m\times n$ integer matrix, let $H$ be the Hermite normal form of
the transpose $A^t$, and let $U$ be an $n\times n$ matrix with $\det(U) = \pm
1$ such that $U A^t = H$. If $r$ is the rank of $H$, then the last $n-r$ rows
of $U$ form a lattice basis for the integer nullspace of $A$.
\end{proposition}

\begin{proof}
Cohen \cite[Proposition 2.4.9]{Cohen}.
\end{proof}

In degree 7, if we use this command from Maple's \texttt{LinearAlgebra}
package,
  \[
  \texttt{HermiteForm( xactionmatrix, output='U' );}
  \]
and extract the last 5 rows of the result, then we obtain an integer basis for
the nullspace with the following (sorted) squared Euclidean norms: 130977928,
264952077, 483975356, 571555922, 1935778474. At this point the results are much
worse than those obtained using the RCF over $\mathbb{Q}$.

\begin{theorem}
For the adjoint representation of $\sltwo$, the space of invariants in degree 7
has dimension 5 and the following basis:
  \allowdisplaybreaks
  \begin{align*}
  &
  I_7^{(1)}
  =
     2  [[[ca][ba]][[ca]c]]
     +  [[[ca][ba]][[cb]b]]
    +2  [[[cb][ba]][[ba]c]]
     -  [[[cb][ba]][[ca]b]]
  \\
  &
     -  [[[cb][ca]][[ba]b]]
    -2  [[[cb][ca]][[ca]a]],
  \\
  &
  I_7^{(2)}
  =
        [[[ca][ba]][[cb]b]]
     +  [[[cb][ba]][[ba]c]]
     -  [[[cb][ba]][[ca]b]]
    -2  [[[cb][ca]][[ca]a]]
  \\
  &
     +  [[[[ba]a][cb]][cb]]
     -  [[[[ba]b][ca]][cb]]
    +2  [[[[ca]a][ca]][cb]]
     +  [[[[ca]b][ba]][cb]]
  \\
  &
    -2  [[[[ca]b][ca]][ca]]
    +2  [[[[ca]c][ba]][ca]]
     -  [[[[cb]b][ba]][ca]]
     +  [[[[cb]c][ba]][ba]],
  \\
  &
  I_7^{(3)}
  =
     2  [[[ca][ba]][[ca]c]]
    -2  [[[ca][ba]][[cb]b]]
    +3  [[[cb][ba]][[ca]b]]
     -  [[[cb][ca]][[ba]b]]
  \\
  &
     -  [[[[ba]a]b][[cb]c]]
    +2  [[[[ba]a]c][[cb]b]]
     +  [[[[ba]b]b][[ca]c]]
    -3  [[[[ba]b]c][[ca]b]]
  \\
  &
    +2  [[[[ba]c]c][[ba]b]]
    -2  [[[[ca]a]a][[cb]c]]
    +2  [[[[ca]a]b][[ca]c]]
     -  [[[[ca]a]b][[cb]b]]
  \\
  &
    -4  [[[[ca]a]c][[ba]c]]
    +2  [[[[ca]a]c][[ca]b]]
     -  [[[[ca]b]b][[ba]c]]
    +2  [[[[ca]b]b][[ca]b]]
  \\
  &
     -  [[[[ca]b]c][[ba]b]]
    -2  [[[[ca]b]c][[ca]a]]
    +2  [[[[ca]c]c][[ba]a]]
     -  [[[[cb]b]b][[ca]a]]
  \\
  &
     +  [[[[cb]b]c][[ba]a]],
  \\
  &
  I_7^{(4)}
  =
     8  [[[ca][ba]][[ca]c]]
     -  [[[ca][ba]][[cb]b]]
     +  [[[cb][ba]][[ba]c]]
    +5  [[[cb][ba]][[ca]b]]
  \\
  &
    +2  [[[cb][ca]][[ca]a]]
    -2  [[[[ba]a]b][[cb]c]]
    -8  [[[[ba]a]c][[ca]c]]
     -  [[[[ba]b]b][[cb]b]]
  \\
  &
    -6  [[[[ba]b]c][[ca]b]]
    -8  [[[[ba]c]c][[ca]a]]
    -4  [[[[ca]a]a][[cb]c]]
    +4  [[[[ca]a]b][[ca]c]]
  \\
  &
    -2  [[[[ca]a]b][[cb]b]]
    +8  [[[[ca]a]c][[ba]c]]
    -4  [[[[ca]a]c][[ca]b]]
    +2  [[[[ca]b]b][[ba]c]]
  \\
  &
    +2  [[[[ca]b]b][[ca]b]]
    +2  [[[[ca]b]c][[ba]b]]
    +4  [[[[ca]b]c][[ca]a]]
    +4  [[[[ca]c]c][[ba]a]]
  \\
  &
     +  [[[[cb]b]b][[ba]b]]
    +2  [[[[cb]b]c][[ba]a]]
     +  [[[[ba]a][cb]][cb]]
     -  [[[[ba]b][ca]][cb]]
  \\
  &
    +2  [[[[ca]a][ca]][cb]]
     +  [[[[ca]b][ba]][cb]]
    -2  [[[[ca]b][ca]][ca]]
    +2  [[[[ca]c][ba]][ca]]
  \\
  &
     -  [[[[cb]b][ba]][ca]]
     +  [[[[cb]c][ba]][ba]],
  \\
  &
  I_7^{(5)}
  =
     2  [[[ca][ba]][[cb]b]]
    -8  [[[cb][ba]][[ba]c]]
    +2  [[[cb][ba]][[ca]b]]
    -6  [[[cb][ca]][[ba]b]]
  \\
  &
    -4  [[[cb][ca]][[ca]a]]
    +4  [[[[ba]a][cb]][cb]]
    +6  [[[[ba]b][ca]][cb]]
    +8  [[[[ba]c][ba]][cb]]
  \\
  &
   +32  [[[[ba]c][ca]][ca]]
   +12  [[[[ca]a][ca]][cb]]
   -12  [[[[ca]b][ca]][ca]]
   -20  [[[[ca]c][ba]][ca]]
  \\
  &
    -8  [[[[cb]b][ba]][ca]]
    -4  [[[[cb]c][ba]][ba]]
    +4  [[[[[ba]a]b]c][cb]]
   +16  [[[[[ba]a]c]c][ca]]
  \\
  &
     +  [[[[[ba]b]b]b][cb]]
    +4  [[[[[ba]b]b]c][ca]]
    +4  [[[[[ba]b]c]c][ba]]
    +8  [[[[[ca]a]a]c][cb]]
  \\
  &
    +2  [[[[[ca]a]b]b][cb]]
    -8  [[[[[ca]a]b]c][ca]]
    -8  [[[[[ca]a]c]c][ba]]
    -2  [[[[[ca]b]b]b][ca]]
  \\
  &
    -6  [[[[[ca]b]b]c][ba]]
     -  [[[[[cb]b]b]b][ba]].
  \end{align*}
\end{theorem}

\begin{proof}
If we use the Maple command
  \[
  \texttt{HermiteForm( xactionmatrix, output='U', method='integer[reduced]' );}
  \]
and extract the last 5 rows of the result, then we obtain an integer basis for
the nullspace with the following (sorted) squared Euclidean norms: 15, 24, 82,
446, 2574.  This is a substantial improvement over the results obtained using
the RCF.  For further details on the application of HNF and LLL to problems in
nonassociative algebra, see Bremner and Peresi \cite{BremnerPeresi}.
\end{proof}

The invariants of degree 8 are calculated explicitly in Hu \cite{Hu}. The
matrix $[X]$ has size $127 \times 136$; by Lemma \ref{surjectivelemma}, its
rank is 127 and so its nullspace has dimension 9. Using the canonical integral
basis of the nullspace obtained from the row canonical form, we obtain the
following sorted list of squared norms of the coefficient vectors of the
invariants: 83, 95, 95, 143, 147, 150, 5030, 18490, 63770. Using the Hermite
normal form with lattice basis reduction gives these squared norms: 32, 47, 47,
62, 83, 143, 1058, 2791, 31295. Again we have found significantly simpler
invariants: the first four invariants using HNF and LLL have coefficient
vectors shorter than the first invariant using the RCF.

We extended these calculations to degree 12, and obtained the results in Table
\ref{sl2adjointdimensions}. The invariants $I_5$ and $I_6$ imply that there is
a 1-dimensional space of non-primitive invariants in degree 11; the invariants
$I_5$ and $I_7^{(1)}, \dots, I_7^{(5)}$ imply that there is a 5-dimensional
space of non-primitive invariants in degree 12. For degree $\ge 9$ we used
modular arithmetic with $p = 101$ in order to ensure that each matrix entry
would only use one byte of memory during the calculation of the row canonical
form of the $x$-action matrix $[X]$.

  \begin{table}
  \begin{center}
  \begin{tabular}{rrrrr}
  degree &\quad weight 0 &\quad weight 2 &\quad invariants &\quad primitive \\
   1 &\quad    1 &\quad    1 &\quad   0 &\quad   0 \\[-2pt]
   2 &\quad    1 &\quad    1 &\quad   0 &\quad   0 \\[-2pt]
   3 &\quad    2 &\quad    2 &\quad   0 &\quad   0 \\[-2pt]
   4 &\quad    4 &\quad    4 &\quad   0 &\quad   0 \\[-2pt]
   5 &\quad   10 &\quad    9 &\quad   1 &\quad   1 \\[-2pt]
   6 &\quad   22 &\quad   21 &\quad   1 &\quad   1 \\[-2pt]
   7 &\quad   56 &\quad   51 &\quad   5 &\quad   5 \\[-2pt]
   8 &\quad  136 &\quad  127 &\quad   9 &\quad   9 \\[-2pt]
   9 &\quad  348 &\quad  323 &\quad  25 &\quad  25 \\[-2pt]
  10 &\quad  890 &\quad  835 &\quad  55 &\quad  55 \\[-2pt]
  11 &\quad 2332 &\quad 2188 &\quad 144 &\quad 143 \\[-2pt]
  12 &\quad 6136 &\quad 5798 &\quad 338 &\quad 333
  \end{tabular}
  \end{center}
  \caption{Dimensions of invariants in the adjoint representation of $\sltwo$}
  \label{sl2adjointdimensions}
  \end{table}


\section{Lie invariants in the natural representation of $\slthree$}
\label{sl3natural}

In this section, we present explicit invariants of minimal degree for
$\slthree$ in the natural representation.  As implied by Wever \cite{Wever1,
Wever2}, there are no invariants in degree $\le 9$, and a 4-dimensional space
of invariants in degree 9.  We find the first explicit basis for this
4-dimensional space.

We use the standard ordered basis of the 8-dimensional simple Lie algebra
$\slthree$:
  \allowdisplaybreaks
  \begin{alignat*}{3}
  x_1
  &=
  \left[
  \begin{array}{rrr}
  0 & 1 & 0 \\
  0 & 0 & 0 \\
  0 & 0 & 0 \\
  \end{array}
  \right],
  &\qquad
  x_2
  &=
  \left[
  \begin{array}{rrr}
  0 & 0 & 0 \\
  0 & 0 & 1 \\
  0 & 0 & 0 \\
  \end{array}
  \right],
  &\qquad
  x_3
  &=
  \left[
  \begin{array}{rrr}
  0 & 0 & 1 \\
  0 & 0 & 0 \\
  0 & 0 & 0 \\
  \end{array}
  \right],
  \\
  h_1
  &=
  \left[
  \begin{array}{rrr}
  1 &\!\!\!\!  0 & 0 \\
  0 &\!\!\!\! -1 & 0 \\
  0 &\!\!\!\!  0 & 0 \\
  \end{array}
  \right],
  &\qquad
  h_2
  &=
  \left[
  \begin{array}{rrr}
  0 & 0 &\!\!\!\!  0 \\
  0 & 1 &\!\!\!\!  0 \\
  0 & 0 &\!\!\!\! -1 \\
  \end{array}
  \right],
  \\
  y_1
  &=
  \left[
  \begin{array}{rrr}
  0 & 0 & 0 \\
  1 & 0 & 0 \\
  0 & 0 & 0 \\
  \end{array}
  \right],
  &\qquad
  y_2
  &=
  \left[
  \begin{array}{rrr}
  0 & 0 & 0 \\
  0 & 0 & 0 \\
  0 & 1 & 0 \\
  \end{array}
  \right],
  &\qquad
  y_3
  &=
  \left[
  \begin{array}{rrr}
  0 & 0 & 0 \\
  0 & 0 & 0 \\
  1 & 0 & 0 \\
  \end{array}
  \right].
  \end{alignat*}
We consider the free Lie algebra $L$ generated by the ordered set $A = \{ a, b,
c \}$ with $a < b < c$. The generators $a, b, c$ are the Hall words of degree
1; they form a basis of the subspace $L_1$.  We regard the 3-dimensional space
$L_1$ as the natural representation of $\slthree$ by identifying $a, b, c$ with
column vectors as follows:
  \[
  a = \left[ \begin{array}{r} 1 \\ 0 \\ 0\end{array} \right],
  \qquad
  b = \left[ \begin{array}{r} 0 \\ 1 \\ 0 \end{array} \right],
  \qquad
  c = \left[ \begin{array}{r} 0 \\ 0 \\ 1 \end{array} \right].
  \]
The natural action of $\slthree$ on $L_1$ by matrix-vector multiplication gives
  \[
  \begin{array}{llllll}
  x_1.a = 0, &\quad x_1.b = a, &\quad x_1.c = 0, &\quad
  x_2.a = 0, &\quad x_2.b = 0, &\quad x_2.c = b,
  \\
  x_3.a = 0, &\quad x_3.b = a, &\quad x_3.c = a,
  &\quad h_1.a = a, &\quad h_1.b = -b, &\quad h_1.c = 0,
  \\
  h_2.a = 0, &\quad h_2.b = b, &\quad h_2.c = -c, &\quad
  y_1.a = b, &\quad y_1.b = 0, &\quad y_1.c = 0,
  \\
  y_2.a = 0, &\quad y_2.b = c, &\quad y_2.c = 0, &\quad
  y_3.a = c, &\quad y_3.b = 0, &\quad y_3.c = 0.
  \end{array}
  \]
We define the action of $\slthree$ on the Hall words of degree $j$, which form
a basis of $L_j$, inductively by degree as in the previous sections. Details
about representations of $\slthree$ may be found in Fulton and Harris
\cite[Lecture 12]{FultonHarris}.

\begin{definition}
By the \emph{weight} $(w_1,w_2)$ of the Hall word $t$, we mean the ordered pair
of its eigenvalues with respect to the action of $h_1$ and $h_2$. The
\emph{weight space} $L_j^{(w_1,w_2)}$ is the subspace of $L_j$ spanned by the
Hall words of weight $(w_1,w_2)$.
\end{definition}

\begin{lemma} \label{sl3weightlemma}
We have
  \[
  L_j^{(w_1,w_2)}
  =
  \mathrm{span} \{ \,
  t
  \mid
  \#a(t) - \#b(t) = w_1, \,
  \#b(t) - \#c(t) = w_2
  \}.
  \]
\end{lemma}

\begin{proof}
From the natural action of $\slthree$ on $L_1$ we see that
  \[
  \mathrm{weight}(a) = (1,0),
  \qquad
  \mathrm{weight}(b) = (-1,1),
  \qquad
  \mathrm{weight}(c) = (0,-1).
  \]
Therefore the weight of an arbitrary Hall word $t$ equals
  \[
  \#a(t)(1,0) + \#b(t)(-1,1) + \#c(t)(0,-1)
  =
  \big( \, \#a(t){-}\#b(t), \, \#b(t){-}\#c(t) \, \big).
  \]
This completes the proof.
\end{proof}

The root system of $\slthree$ with respect to the Cartan subalgebra spanned by
$h_1, h_2$ is the same as the weight system of the adjoint representation. The
adjoint action of $h_1$ and $h_2$ on the simple root vectors $x_1$ and $x_2$ is
as follows:
  \[
  [ h_1, x_1 ] = 2x_1, \qquad
  [ h_1, x_2 ] = -x_2, \qquad
  [ h_2, x_1 ] = -x_1, \qquad
  [ h_2, x_2 ] = 2x_2, \qquad
  \]
From this we see that the weights of $x_1$ and $x_2$ are $(2,-1)$ and $(-1,2)$
respectively. It follows that if $M$ is any representation of $\slthree$ and $v
\in M$ has weight $(0,0)$, then $x_1.v$ and $x_2.v$ have weights $(2,-1)$ and
$(-1,2)$ respectively.

Our problem is to determine which linear combinations of the Hall words of
degree $j$ are invariant under the natural action of $\slthree$. That is, we
want to find the elements $Z \in L_j$ which satisfy $D.Z = 0$ for all $D \in
\slthree$.

\begin{lemma} \label{sl3kernellemma}
The Lie invariants of degree $j$ for $\slthree$ in the natural representation
are the kernel of the linear map
  \[
  X \colon L_j^{(0,0)} \to L_j^{(2,-1)} \oplus L_j^{(-1,2)},
  \qquad
  X(Z) = ( x_1.Z, x_2.Z ).
  \]
\end{lemma}

\begin{proof}
Since $L_j$ is finite-dimensional and $\slthree$ is a semisimple Lie algebra,
Weyl's Theorem implies that $L_j$ is isomorphic to a direct sum of simple
highest weight modules. If $Z \in L_j$ satisfies $D.Z = 0$ for all $D \in
\slthree$, then clearly $h_1.Z = h_2.Z = 0$, and so $Z \in L_j^{(0,0)}$, the
weight space of weight $(0,0)$. If $D \ne \mathrm{span}( h_1, h_2 )$, then it
suffices to consider $D = x_i$ and $D = y_i$ for $i = 1, 2, 3$. If $h_i . Z =
x_i . Z = 0$ then also $y_i . Z = 0$, as we saw in the proof of Lemma
\ref{xactionkernel}; so it suffices to consider $x_i$ ($i = 1, 2, 3$).  But in
$\slthree$ we have $[ x_1, x_2 ] = x_3$, and so if $x_1 . Z = x_2 . Z = 0$ then
also $x_3 . Z = 0$. This completes the proof.
\end{proof}

To compute explicit invariants, we consider separately the linear maps
  \[
  X_1 \colon L_j^{(0,0)} \to L_j^{(2,-1)},
  \qquad
  X_2 \colon L_j^{(0,0)} \to L_j^{(-1,2)}.
  \]
The matrices $[X_1]$ and $[X_2]$ represent the linear maps $X_1$ and $X_2$. We
stack $[X_1]$ on top of $[X_2]$ to get the matrix $[X]$ representing the linear
map $X$ of Lemma \ref{sl3kernellemma}. The coefficient vectors of the
invariants are then the vectors in the nullspace of $[X]$.

By Wever \cite{Wever1, Wever2}, we know that there are no invariants in degree
$j$ unless $j$ is a multiple of 3.  We can also see this as follows: Lemma
\ref{sl3weightlemma} implies that any Hall word $t$ of weight $(0,0)$ must
satisfy $\#a(t) = \#b(t) = \#c(t)$ and hence its degree $j = \#a(t) + \#b(t) +
\#c(t)$ must be a multiple of 3. The next two results confirm Wever's
observations for $q = 3$ and $d = 3, 6$.

\begin{lemma}
For the natural representation of $\slthree$, there are no Lie invariants in
degree 3.
\end{lemma}

\begin{proof}
In degree $3$, there are two Hall words of weight $(0,0)$, namely $[[b,a],c]$
and $[[c,a],b]$.  There is one Hall word of weight $(2,-1)$, namely
$[[c,a],a]$, and one Hall word of weight $(-1,2)$, namely $[[b,a],b]]$. We
calculate that
  \[
  x_1 . [[b{,}a]{,}c] = 0,
  \;\;
  x_1 . [[c{,}a]{,}b] = [[c{,}a]{,}a],
  \;\;
  x_2 . [[b{,}a]{,}c] = [[b{,}a]{,}b],
  \;\;
  x_2 . [[c{,}a]{,}b] = [[b{,}a]{,}b].
  \]
From this we obtain the matrices
  \[
  [X_1] = \left[ \begin{array}{rr} 0 & 1 \end{array} \right],
  \qquad
  [X_2] = \left[ \begin{array}{rr} 1 & 1 \end{array} \right],
  \qquad
  [X] =
  \left[
  \begin{array}{rr}
  0 & 1 \\
  1 & 1 \\
  \end{array}
  \right].
  \]
It is clear that $[X]$ has rank 2 and hence its nullspace is $\{0\}$.
\end{proof}

  \begin{table}
  \small
  \[
  \begin{array}{lllll}
  {[[[ca]b][[ba]c]]}^{9} &
  {[[[ca]c][[ba]b]]}^{12} &
  {[[[cb]b][[ca]a]]}^{19} &
  {[[[cb]c][[ba]a]]}^{22} &
  {[[[ca][ba]][cb]]}^{31} \\
  {[[[cb][ba]][ca]]}^{33} &
  {[[[[ba]a]c][cb]]}^{45} &
  {[[[[ba]b]c][ca]]}^{50} &
  {[[[[ba]c]c][ba]]}^{52} &
  {[[[[ca]a]b][cb]]}^{60} \\
  {[[[[ca]b]b][ca]]}^{65} &
  {[[[[ca]b]c][ba]]}^{67} &
  {[[[[[ba]a]b]c]c]}^{90} &
  {[[[[[ca]a]b]b]c]}^{104}
  \\
  &&&
  \\
  {[[[ca]a][[ba]c]]}^{6} &
  {[[[ca]b][[ca]a]]}^{10} &
  {[[[ca]c][[ba]a]]}^{11} &
  {[[[ca][ba]][ca]]}^{30} &
  {[[[[ba]a]c][ca]]}^{44} \\
  {[[[[ca]a]a][cb]]}^{57} &
  {[[[[ca]a]b][ca]]}^{59} &
  {[[[[ca]a]c][ba]]}^{61} &
  {[[[[[ba]a]a]c]c]}^{87} &
  {[[[[[ca]a]a]b]c]}^{101}
  \\
  &&&
  \\
  {[[[ba]c][[ba]b]]}^{3} &
  {[[[ca]b][[ba]b]]}^{8} &
  {[[[cb]b][[ba]a]]}^{16} &
  {[[[cb][ba]][ba]]}^{32} &
  {[[[[ba]a]b][cb]]}^{42} \\
  {[[[[ba]b]b][ca]]}^{47} &
  {[[[[ba]b]c][ba]]}^{49} &
  {[[[[ca]b]b][ba]]}^{64} &
  {[[[[[ba]a]b]b]c]}^{89} &
  {[[[[[ca]a]b]b]b]}^{103}
  \end{array}
  \]
  \caption{Hall words of weights $(0,0)$, $(2,-1)$, $(-1,2)$ in degree 6}
  \label{sl3hallwords6}
  \end{table}

  \begin{table}
  \[
  \left[
  \begin{array}{rrrrrrrrrrrrrrrrrrrr}
  1 & . & . & . & . & . & . & . & . & . & -1 & 1 & . &  . & . & . & . & . & . & . \\
  . & . & 1 & . & . & . & . & . & . & . &  1 & 1 & . &  . & . & . & . & . & . & . \\
  1 & 2 & . & . & . & . & . & . & . & . &  . & . & 1 &  . & . & . & . & . & . & . \\
  . & . & 1 & . & . & . & . & . & . & . &  . & . & 1 &  . & . & . & . & . & . & . \\
  . & . & . & 1 & . & . & . & . & . & . &  . & . & . &  . & . & . & . & . & . & . \\
  . & . & . & 1 & . & . & . & . & . & . &  . & . & . &  1 & . & . & . & . & . & . \\
  . & . & . & . & 1 & . & . & . & . & . &  . & . & . &  . & 1 & . & . & . & . & . \\
  . & . & . & . & 1 & . & . & . & . & . &  . & . & . &  . & . & 1 & 1 & . & . & . \\
  . & . & . & . & . & . & . & . & . & . &  . & . & . & -1 & . & . & 2 & . & . & . \\
  . & . & . & . & . & 1 & 1 & . & . & . &  . & . & . &  . & 1 & . & . & . & . & . \\
  . & . & . & 1 & . & . & 2 & . & . & . &  . & . & . &  . & . & 1 & . & 1 & . & . \\
  . & . & . & . & . & . & . & 1 & . & . &  . & . & . &  . & . & . & 1 & 1 & . & . \\
  . & . & . & . & . & . & . & . & 1 & . &  . & . & . &  . & 1 & . & . & . & 2 & . \\
  1 & . & . & . & . & . & . & 1 & . & 2 &  . & . & . &  . & . & . & . & . & 1 & 1
  \end{array}
  \right]
  \]
  \caption{The transpose of the matrix $[X]$ for $\slthree$ in degree 6}
  \label{sl3matrix6}
  \end{table}

\begin{lemma}
For the natural representation of $\slthree$, there are no Lie invariants in
degree 6.
\end{lemma}

\begin{proof}
In degree 6, there are 14 Hall words in $a,b,c$ for weight $(0,0)$, and 10 Hall
words for each weight $(2,-1)$ and $(-1,2)$; see Table \ref{sl3hallwords6}. For
each Hall word of weight $(0,0)$ we compute the actions of $x_1$ and $x_2$ and
use Hall's algorithm (Figure \ref{hallalgorithm}) to reduce the results to
linear combinations of Hall words of weights $(2,-1)$ and $(-1,2)$
respectively.  The coefficients give the entries of the $20 \times 14$ matrix
$[X]$; see Table \ref{sl3matrix6} (but note that we have displayed the
transpose).  This matrix has rank 14, and hence its nullspace is $\{ 0 \}$.
\end{proof}

Wever \cite{Wever1} found an invariant in degree 9 for the natural
representation of $\slthree$:
  \[
  [[[[ab][ac]][[ab][bc]]]c]
  +
  [[[[bc][ba]][[bc][ca]]]a]
  +
  [[[[ca][cb]][[ca][ab]]]b].
  \]
Our next result presents the first complete explicit basis for the space of Lie
invariants for the natural representation of $\slthree$ in the minimal degree
9.  We note that the terms of Wever's invariant are not Hall words; the Hall
form of his invariant is just the negative of $I_9^{(1)}$ below.  (It would be
interesting to see if the other invariants $I_9^{(2)}, \dots, I_9^{(4)}$ can be
expressed more compactly without using Hall words.)

\begin{theorem}
For the natural representation of $\slthree$, the space of invariants in degree
9 has dimension 4 and the following basis:
  \allowdisplaybreaks
  \begin{align*}
  I_9^{(1)}
  &=
     [[[[ba]a][cb]][[cb][ca]]]  
  -  [[[[ba]b][ca]][[cb][ca]]]  
  +  [[[[ba]c][ba]][[cb][ca]]]  
  \\
  &\quad
  -  [[[[ba]c][cb]][[ca][ba]]]  
  -  [[[[ca]a][cb]][[cb][ba]]]  
  +  [[[[ca]b][ca]][[cb][ba]]]  
  \\
  &\quad
  +  [[[[ca]b][cb]][[ca][ba]]]  
  -  [[[[ca]c][ba]][[cb][ba]]]  
  -  [[[[cb]b][ca]][[ca][ba]]]  
  \\
  &\quad
  +  [[[[cb]c][ba]][[ca][ba]]], 
  \\
  I_9^{(2)}
  &=
     [[[[ca][ba]][ba]][[cb]c]]  
  -  [[[[ca][ba]][ca]][[cb]b]]  
  -2 [[[[ca][ba]][cb]][[ba]c]]  
  \\
  &\quad
  +  [[[[ca][ba]][cb]][[ca]b]]  
  -  [[[[cb][ba]][ba]][[ca]c]]  
  +  [[[[cb][ba]][ca]][[ba]c]]  
  \\
  &\quad
  +  [[[[cb][ba]][ca]][[ca]b]]  
  -  [[[[cb][ba]][cb]][[ca]a]]  
  -  [[[[cb][ca]][ca]][[ba]b]]  
  \\
  &\quad
  +  [[[[cb][ca]][cb]][[ba]a]], 
  \\
  I_9^{(3)}
  &=
     [[[[ca][ba]][ba]][[cb]c]]  
  -  [[[[ca][ba]][ca]][[cb]b]]  
  +  [[[[ca][ba]][cb]][[ba]c]]  
  \\
  &\quad
  -2 [[[[ca][ba]][cb]][[ca]b]]  
  -  [[[[cb][ba]][ba]][[ca]c]]  
  -  [[[[cb][ba]][ca]][[ba]c]]  
  \\
  &\quad
  +2 [[[[cb][ba]][ca]][[ca]b]]  
  +  [[[[cb][ba]][cb]][[ca]a]]  
  -  [[[[[ba]a]a][cb]][[cb]c]]  
  \\
  &\quad
  +  [[[[[ba]a]b][ca]][[cb]c]]  
  +  [[[[[ba]a]b][cb]][[ca]c]]  
  -  [[[[[ba]a]c][ba]][[cb]c]]  
  \\
  &\quad
  -2 [[[[[ba]a]c][cb]][[ba]c]]  
  +  [[[[[ba]a]c][cb]][[ca]b]]  
  -  [[[[[ba]b]b][ca]][[ca]c]]  
  \\
  &\quad
  +  [[[[[ba]b]c][ba]][[ca]c]]  
  +2 [[[[[ba]b]c][ca]][[ba]c]]  
  -  [[[[[ba]b]c][ca]][[ca]b]]  
  \\
  &\quad
  -  [[[[[ba]b]c][cb]][[ca]a]]  
  -2 [[[[[ba]c]c][ba]][[ba]c]]  
  +  [[[[[ba]c]c][ba]][[ca]b]]  
  \\
  &\quad
  +  [[[[[ba]c]c][cb]][[ba]a]]  
  +  [[[[[ca]a]a][cb]][[cb]b]]  
  -  [[[[[ca]a]b][ca]][[cb]b]]  
  \\
  &\quad
  +  [[[[[ca]a]b][cb]][[ba]c]]  
  -2 [[[[[ca]a]b][cb]][[ca]b]]  
  +  [[[[[ca]a]c][ba]][[cb]b]]  
  \\
  &\quad
  +  [[[[[ca]a]c][cb]][[ba]b]]  
  -  [[[[[ca]b]b][ca]][[ba]c]]  
  +2 [[[[[ca]b]b][ca]][[ca]b]]  
  \\
  &\quad
  +  [[[[[ca]b]b][cb]][[ca]a]]  
  +  [[[[[ca]b]c][ba]][[ba]c]]  
  -2 [[[[[ca]b]c][ba]][[ca]b]]  
  \\
  &\quad
  -  [[[[[ca]b]c][ca]][[ba]b]]  
  -  [[[[[ca]b]c][cb]][[ba]a]]  
  +  [[[[[ca]c]c][ba]][[ba]b]]  
  \\
  &\quad
  -  [[[[[cb]b]b][ca]][[ca]a]]  
  +  [[[[[cb]b]c][ba]][[ca]a]]  
  +  [[[[[cb]b]c][ca]][[ba]a]]  
  \\
  &\quad
  -  [[[[[cb]c]c][ba]][[ba]a]], 
  \\
  I_9^{(4)}
  &=
     [[[[ba]a][ba]][[[cb]c]c]]  
  -  [[[[ba]a][ca]][[[cb]b]c]]  
  -  [[[[ba]a][cb]][[[ba]c]c]]  
  \\
  &\quad
  +  [[[[ba]a][cb]][[[ca]b]c]]  
  -  [[[[ba]b][ba]][[[ca]c]c]]  
  +  [[[[ba]b][ca]][[[ca]b]c]]  
  \\
  &\quad
  -  [[[[ba]b][cb]][[[ca]a]c]]  
  -  [[[[ba]c][ba]][[cb][ca]]]  
  +2 [[[[ba]c][ba]][[[ba]c]c]]  
  \\
  &\quad
  -  [[[[ba]c][ba]][[[ca]b]c]]  
  +2 [[[[ba]c][ca]][[cb][ba]]]  
  -2 [[[[ba]c][ca]][[[ba]b]c]]  
  \\
  &\quad
  +  [[[[ba]c][ca]][[[ca]b]b]]  
  -2 [[[[ba]c][cb]][[ca][ba]]]  
  +2 [[[[ba]c][cb]][[[ba]a]c]]  
  \\
  &\quad
  -  [[[[ba]c][cb]][[[ca]a]b]]  
  -  [[[[ca]a][ba]][[[cb]b]c]]  
  +  [[[[ca]a][ca]][[[cb]b]b]]  
  \\
  &\quad
  -  [[[[ca]a][cb]][[cb][ba]]]  
  +  [[[[ca]a][cb]][[[ba]b]c]]  
  -  [[[[ca]a][cb]][[[ca]b]b]]  
  \\
  &\quad
  -  [[[[ca]b][ba]][[[ba]c]c]]  
  +2 [[[[ca]b][ba]][[[ca]b]c]]  
  -2 [[[[ca]b][ca]][[cb][ba]]]  
  \\
  &\quad
  +  [[[[ca]b][ca]][[[ba]b]c]]  
  -2 [[[[ca]b][ca]][[[ca]b]b]]  
  +2 [[[[ca]b][cb]][[ca][ba]]]  
  \\
  &\quad
  -  [[[[ca]b][cb]][[[ba]a]c]]  
  +2 [[[[ca]b][cb]][[[ca]a]b]]  
  +  [[[[ca]c][ba]][[cb][ba]]]  
  \\
  &\quad
  -  [[[[ca]c][ba]][[[ba]b]c]]  
  +  [[[[ca]c][ca]][[[ba]b]b]]  
  -  [[[[ca]c][cb]][[[ba]a]b]]  
  \\
  &\quad
  -  [[[[cb]b][ba]][[[ca]a]c]]  
  +  [[[[cb]b][ca]][[ca][ba]]]  
  +  [[[[cb]b][ca]][[[ca]a]b]]  
  \\
  &\quad
  -  [[[[cb]b][cb]][[[ca]a]a]]  
  -  [[[[cb]c][ba]][[ca][ba]]]  
  +  [[[[cb]c][ba]][[[ba]a]c]]  
  \\
  &\quad
  -  [[[[cb]c][ca]][[[ba]a]b]]  
  +  [[[[cb]c][cb]][[[ba]a]a]]. 
  \end{align*}
\end{theorem}

\begin{proof}
We use Maple to generate the 2184 Hall words in $a,b,c$ in degree 9.  There are
186 for weight $(0,0)$, and 140 each for weights $(2,-1)$ and $(-1,2)$.  We
apply $x_1$ and $x_2$, compute the normal forms of the results, and store the
coefficients in the $280 \times 186$ matrix $[X]$.  We compute the row
canonical form of this matrix, and find that the rank is 182; hence the
nullspace has dimension 4.  The canonical basis vectors for the nullspace have
integer components, and their (sorted) squared norms are 10, 13, 64, 79.  Using
the Hermite normal form with lattice basis reduction provides a little
improvement in the longest coefficient vector: we obtain an integral basis for
the nullspace consisting of vectors with (sorted) squared norms 10, 13, 64, 71.
These coefficient vectors correspond to the stated invariants.
\end{proof}

Since the minimal degree of the (primitive) invariants is 9, it follows that
the minimal degree of the non-primitive invariants is 18.  Hence all the
invariants in degrees 12 and 15 are primitive.

\begin{theorem}
For the natural representation of $\slthree$, the space of invariants in degree
12 has dimension 35.
\end{theorem}

\begin{proof}
In degree 12, there are 44220 Hall words; 2880 of weight $(0,0)$ and 2310 of
each weight $(2,-1)$ and $(-1,2)$. Hence the matrix $[X]$ has size $4620 \times
2880$.  Using modular arithmetic with $p = 101$ we use the Maple command
\texttt{RowReduce} from the \texttt{LinearAlgebra[Modular]} package to compute
the row canonical form. The rank is 2845 and hence the nullspace has dimension
35.  We use the command \texttt{Basis} to extract the canonical basis of the
nullspace; the sorted list of the number of terms in each basis invariant is as
follows: 22, 24, 30, 32, 34, 34, 36, 36, 41, 42, 44, 48, 64, 79, 84, 89, 96,
109, 133, 137, 153, 155, 158, 161, 163, 165, 169, 173, 174, 174, 216, 244, 334,
399, 425. The simplest invariant is displayed in Table \ref{sl3degree12}.
\end{proof}

  \begin{table}
  \small
  \begin{align*}
  &
         [[[[ca]c][[ba]b]][[[ca]b][[ba]c]]]  
    {+}  [[[[ca]c][[ba]c]][[[ba]c][[ba]b]]]  
    {-}  [[[[ca]c][[ca]b]][[[ca]b][[ba]b]]]  
  \\[+3pt]
  &
    {-}  [[[[cb]b][[ba]a]][[[ca]c][[ba]c]]]  
    {+}  [[[[cb]b][[ba]a]][[[ca]c][[ca]b]]]  
    {-}  [[[[cb]b][[ba]b]][[[ca]c][[ca]a]]]  
  \\[+3pt]
  &
    {+}  [[[[cb]b][[ba]c]][[[ca]a][[ba]c]]]  
    {+}  [[[[cb]b][[ba]c]][[[ca]b][[ca]a]]]  
    {+}  [[[[cb]b][[ba]c]][[[ca]c][[ba]a]]]  
  \\[+3pt]
  &
    {+}  [[[[cb]b][[ca]a]][[[ca]b][[ba]c]]]  
    {-}  [[[[cb]b][[ca]b]][[[ca]a][[ba]c]]]  
    {-}  [[[[cb]b][[ca]c]][[[ca]b][[ba]a]]]  
  \\[+3pt]
  &
    {+}  [[[[cb]c][[ba]a]][[[ca]b][[ba]c]]]  
    {-}  [[[[cb]c][[ba]b]][[[ca]a][[ba]c]]]  
    {-}  [[[[cb]c][[ba]c]][[[ca]b][[ba]a]]]  
  \\[+3pt]
  &
    {+}  [[[[cb]c][[ca]a]][[[ba]c][[ba]b]]]  
    {-}  [[[[cb]c][[ca]a]][[[ca]b][[ba]b]]]  
    {-}  [[[[cb]c][[ca]b]][[[ba]c][[ba]a]]]  
  \\[+3pt]
  &
    {+}  [[[[cb]c][[ca]b]][[[ca]a][[ba]b]]]  
    {+}  [[[[cb]c][[ca]b]][[[ca]b][[ba]a]]]  
    {+}  [[[[cb]c][[ca]c]][[[ba]b][[ba]a]]]  
  \\[+3pt]
  &
    {-}  [[[[cb]c][[cb]b]][[[ca]a][[ba]a]]]  
  \end{align*}
  \caption{The simplest invariant for $\slthree$ in degree 12}
  \label{sl3degree12}
  \end{table}


\section{Conclusion}

\subsection{Generalized Witt dimension formula}

The formula of Witt \cite{Witt1} gives the dimension of the space of
homogeneous elements of degree $n$ in the free Lie algebra on $q$ generators
each of degree 1. The generalized Witt formula of Kang and Kim \cite{KangKim}
and Jurisich \cite{Jurisich} gives the dimension of the space of homogeneous
elements of vector degree $\mathbf{n} \in \mathbb{Z}^k$ in the free Lie algebra
on any set of generators with given vector degrees in $\mathbb{Z}^k$.  In
particular, we can use this more general formula to find the dimension of the
space of homogeneous elements of degree $n$ in the free Lie algebra on $q_i$
generators in degree $i$.

Let $L$ be the free Lie algebra on $q$ generators over a field $F$, and let $G$
be any simple Lie algebra over $F$ with an irreducible representation of
dimension $q$.  The action of $G$ extends naturally to all of $L$, and we can
consider the free Lie subalgebra $J \subseteq L$ of $G$-invariants.  We have
decompositions of both $L$ and $J$ into direct sums of homogeneous subspaces:
  \[
  L = \bigoplus_{n \ge 1} L_n,
  \qquad
  J = \bigoplus_{n \ge 1} J_n,
  \qquad
  J_n = J \cap L_n.
  \]
Up to a certain degree, all of the invariants will be primitive; that is, not
in the Lie subalgebra generated by known invariants of lower degree.  Suppose
that we have a set of primitive invariants which forms a basis for the space of
all invariants of degree $\le n$.  We can then use the generalized Witt formula
to compute the dimension of the invariants in all degrees $>n$ which belong to
the Lie subalgebra generated by the known primitive invariants.  We can then
compare this dimension to the dimension of the space of all invariants in the
least degree $N > n$ for which invariants are known to exist.  The difference
between these two dimensions will be the number of new primitive invariants in
degree $N$.

A closely related problem is to find an inverse of the Witt dimension formula:

\begin{openproblem}
Let $L = \bigoplus_{n \ge 1} L_n$ be a free Lie algebra, and let $J =
\bigoplus_{n \ge 1} J_n$ be a homogeneous subalgebra; that is $J_n = J \cap
L_n$ for all $n \ge 1$.  Suppose that we know $\dim J_n$ for all $n \ge 1$.  We
know that $J$ is a free Lie algebra.  Determine the degrees of the elements of
a free generating set for $J$.
\end{openproblem}

A solution to this problem would give a way of determining the dimensions of
the spaces of primitive invariants, given the dimensions of the spaces of all
invariants.
\subsection{Dimension formula of Thrall and Brandt}

We can compute the dimension of the space of all invariants in degree $N$ using
the computational methods of this paper, at least for relatively small values
of $N$.  However, in the case of the natural representation of
$\mathfrak{sl}_n(\mathbb{C})$, this dimension is given by a formula of Thrall
\cite{Thrall} and Brandt \cite{Brandt}; see also Reutenauer \cite[\S
8.6.1]{Reutenauer}. The dimension of $J_N$ where $N = mq$ equals the
multiplicity in the $S_N$-module $M_N$ of the simple $S_N$-module corresponding
to the partition $m^q$; here $S_N$ is the symmetric group and $M_N$ is the
space of multilinear Lie polynomials of degree $N$.

\begin{openproblem}
Generalize the formula of Thrall and Brandt to irreducible representations of
$\mathfrak{sl}_n(\mathbb{C})$ other than the natural $n$-dimensional
representation, and then to arbitrary irreducible representations of arbitrary
simple Lie algebras.
\end{openproblem}

We outline one possible approach to this problem.  Let $L_n^{\mathbf{0}}$ be
the span of the Hall words of degree $n$ and weight $\mathbf{0}$, and let
$L_n^{\mathbf{w}_i}$ be the span of the Hall words of degree $n$ and weight
$\mathbf{w}_i$, where $\mathbf{w}_i$ is the weight of the simple root vector
$x_i$ in the adjoint representation of the simple Lie algebra $A$.  (For
example, consider the computations for $A = \slthree$ in Section
\ref{sl3natural}.)  Lemma \ref{surjectivelemma} shows that the $x_i$-action
maps $X_i\colon L_n^{\mathbf{0}} \to L_n^{\mathbf{w}_i}$ are surjective, and it
may possible to use this fact to compute the dimension of the kernel of the map
$X\colon L_n^{\mathbf{0}} \to \bigoplus_i L_n^{\mathbf{w}_i}$. To illustrate,
we consider the three cases studied in this paper; the formula of Thrall and
Brandt can be applied to the first and third cases, but not to the second:
  \begin{itemize}
  \item $\sltwo$ in the natural representation: $\deg(a) = (1,1)$, $\deg(b)
      = (1,-1)$. The generalized Witt formula gives the dimension of the
      space of degree $(n,w)$ spanned by the Hall words of degree $n = \#a
      + \#b$, weight $w = \#a - \#b$.
  \item $\sltwo$ in the adjoint representation: $\deg(a) = (1,2)$, $\deg(b)
      = (1,0)$, $\deg(c) = (1,-2)$. The space of degree $(n,w)$ is spanned
      by the Hall words of degree $n = \#a + \#b + \#c$, $w = 2\#a - 2\#c$.
  \item $\slthree$ in the natural representation: $\deg(a) = (1,1,0)$,
      $\deg(b) = (1,-1,1)$, $\deg(c) = (1,0,-1)$. The subspace of degree
      $(n,w_1,w_2)$ is spanned by the Hall words of degree $n = \#a + \#b +
      \#c$, weight $w_1 = \#a - \#b$, $w_2 = \#b - \#c$.
  \end{itemize}

\subsection{Finite generation problem}

The computations in this paper suggest that the dimensions of the spaces of
primitive invariants grow rapidly as a function of the degree.  This raises the
following question.

\begin{openproblem}
Let $L$ be the free Lie algebra on $q$ generators over a field $F$, and let $A$
be any simple Lie algebra over $F$ with an irreducible representation of
dimension $q$.  Let $j \subseteq L$ be the subalgebra of Lie invariants for the
action of $A$ on $L$. Prove or disprove that $J$ is finitely generated.
\end{openproblem}

One way to approach this problem would be to use the dimension formulas to
study asymptotic estimates for the number of invariants in each degree.


\section*{Acknowledgements}

This work forms part of the M.Sc.~thesis of the second author, written under
the supervision of the first author.  The first author was partially supported
by NSERC, the Natural Sciences and Engineering Research Council of Canada.



\begin{thebibliography}{99}

\bibitem{AdkinsWeintraub}
  \textsc{W. A. Adkins, S. H. Weintraub}:
  \emph{Algebra: An Approach via Module Theory}.
  Graduate Texts in Mathematics, 136.
  Springer, New York, 1992.
  MR1181420 (94a:00001)

\bibitem{Bahturin}
  \textsc{Y. A. Bahturin}:
  \emph{Identical Relations in Lie Algebras}.
  VNU Science Press, Utrecht, 1987.
  MR0886063 (88f:17032)

\bibitem{Brandt}
  \textsc{A. Brandt}:
  The free Lie ring and Lie representations of the full linear group.
  \emph{Trans. Amer. Math. Soc.}
  56 (1944) 528--536.
  MR0011305 (6,146d)

\bibitem{Bremner}
  \textsc{M. R. Bremner}:
  Lie invariants of degree ten.
  \emph{Int. J. Math. Game Theory Algebra}
  8 (1998), no.~2-3, 115--122.
  MR1641779 (99h:17004)

\bibitem{BremnerPeresi}
  \textsc{M. R. Bremner, L. A. Peresi}:
  An application of lattice basis reduction to polynomial identities for
  algebraic structures.
  \emph{Linear Algebra Appl.}
  430 (2009), no.~2-3, 642--659.
  MR2469318 (2009i:17003)

\bibitem{Burrow1}
  \textsc{M. D. Burrow}:
  Invariants of free Lie rings.
  \emph{Comm. Pure Appl. Math.}
  11 (1958) 419--431.
  MR0099357 (20 \#5797)

\bibitem{Burrow2}
  \textsc{M. D. Burrow}:
  The enumeration of Lie invariants.
  \emph{Comm. Pure Appl. Math.}
  20 (1967) 401--411.
  MR0210766 (35 \#1652)

\bibitem{Cohen}
  \textsc{H. Cohen}:
  \emph{A Course in Computational Algebraic Number Theory}.
  Graduate Texts in Mathematics, 138.
  Springer, Berlin, 1993.
  MR1228206 (94i:11105)

\bibitem{FultonHarris}
  \textsc{W. Fulton, J. Harris}:
  \emph{Representation Theory: A First Course}.
  Graduate Texts in Mathematics, 129.
  Springer, New York, 1991.
  MR1153249 (93a:20069)

\bibitem{Hall}
  \textsc{M. Hall, Jr.}:
  A basis for free Lie rings and higher commutators in free groups.
  \emph{Proc. Amer. Math. Soc.}
  1 (1950) 575--581.
  MR0038336 (12,388a)

\bibitem{Hu}
  \textsc{Jiaxiong Hu}:
  \emph{Invariant Lie Polynomials in Two and Three Variables}.
  M.Sc. Thesis,
  Department of Mathematics and Statistics,
  University of Saskatchewan, Canada,
  August 2009.

\bibitem{Humphreys}
  \textsc{J. E. Humphreys}:
  \emph{Introduction to Lie Algebras and Representation Theory}.
  Graduate Texts in Mathematics, 9.
  Springer, New York, 1972.
  MR0323842 (48 \#2197)

\bibitem{Jurisich}
  \textsc{E. Jurisich}:
  Generalized Kac-Moody Lie algebras, free Lie algebras and the structure of the
  monster Lie algebra.
  \emph{J. Pure Appl. Algebra}
  126 (1998), no.~1-3, 233--266.
  MR1600542 (99b:17032)

\bibitem{KangKim}
  \textsc{S. J. Kang, M. H. Kim}:
  Free lie algebras, generalized Witt formula, and the denominator identity.
  \emph{J. Algebra}
  183 (1996), no.~2, 560--594.
  MR1399040 (97e:17042)

\bibitem{Magnus}
  \textsc{W. Magnus}:
  \"Uber Gruppen und zugeordnete Liesche Ringe.
  \emph{J. Reine Angew. Math.}
  182 (194) 142--149.
  MR0003411 (2,214d)

\bibitem{Reutenauer}
  \textsc{C. Reutenauer}:
  \emph{Free Lie Algebras}.
  London Mathematical Society Monographs, 7.
  Oxford University Press, New York, 1993.
  MR1231799 (94j:17002)

\bibitem{Shirshov}
  \textsc{A. I. Shirshov}:
  Subalgebras of free Lie algebras.
  \emph{Mat. Sbornik N.S.}
  33(75), (1953) 441--452.
  MR0059892 (15,596d)

\bibitem{ShirshovWorks}
  \textsc{A. I. Shirshov}:
  \emph{Selected Works of A. I. Shirshov}.
  Translated from the Russian by Murray R. Bremner and Mikhail V. Kotchetov.
  Edited by Leonid A. Bokut, Victor Latyshev, Ivan Shestakov and Efim Zelmanov.
  Contemporary Mathematicians.
  Birkhäuser Verlag, Basel, 2009
  MR2547481

\bibitem{Thrall}
  \textsc{R. M. Thrall}:
  On symmetrized Kronecker powers and the structure of the free Lie ring.
  \emph{Amer. J. Math.}
  64 (1942) 371--388.
  MR0006149 (3,262d)

\bibitem{MCA}
  \textsc{J. von zur Gathen, J. Gerhard}:
  \emph{Modern Computer Algebra}.
  Second edition.
  Cambridge University Press, 2003.
  MR2001757 (2004g:68202)

\bibitem{Wever1}
  \textsc{F. Wever}:
  \"{U}ber Invarianten in Lie'schen Ringen.
  \emph{Math. Ann.} 120 (1949) 563--580.
  MR0029893 (10,676e)

\bibitem{Wever2}
  \textsc{F. Wever}:
  Operatoren in Lieschen Ringen.
  \emph{J. Reine Angew. Math.}
  187 (1949) 44--55.
  MR0034397 (11,579i)

\bibitem{Witt1}
  \textsc{E. Witt}:
  Treue Darstellungen beliebiger Liescher Ringe.
  \emph{Collectanea Math.}
  6 (1953) 107--114.
  MR0062717 (16,5f)

\bibitem{Witt2}
  \textsc{E. Witt}:
  Die Unterringe der freien Lieschen Ringe.
  \emph{Math. Z.}
  64 (1956) 195--216.
  MR0077525 (17,1050a)

\end{thebibliography}
\end{document}